\documentclass[11pt]{amsart}

\usepackage{amssymb,amsmath,amscd,amsthm,enumerate,hyperref,verbatim}
\usepackage[utf8]{inputenc}
\usepackage[mathscr]{euscript}
\usepackage[left=1.2in,right=1.2in,top=1in,bottom=1in]{geometry}

\DeclareMathOperator{\Hom}{Hom}

\renewcommand{\Im}{\operatorname{Im}}
\renewcommand{\Re}{\operatorname{Re}}

\newcommand{\RNum}[1]{\uppercase\expandafter{\romannumeral #1\relax}}

\newcommand{\sslash}{\mathbin{/\mkern-6mu/}}

\newcommand{\norm}[1]{\left\lVert#1\right\rVert}

\theoremstyle{plain}
\newtheorem{corx}{Corollary}

\newtheorem{thmx}[corx]{Theorem}

\newtheorem{thm}{Theorem}[section]
\newtheorem*{thm*}{Theorem}
\newtheorem{lem}[thm]{Lemma}
\newtheorem{prop}[thm]{Proposition}
\newtheorem{cor}[thm]{Corollary}

\theoremstyle{definition}
\newtheorem{defn}{Definition}[section]

\newtheorem*{rem}{Remark}

\numberwithin{equation}{section}
\numberwithin{thm}{section}
\numberwithin{defn}{section}

\begin{document}

\author[C.\ Ouyang]{Charles Ouyang}
\address{Department of Mathematics, Rice University, Houston, TX~77005, USA}
\email{charles.ouyang@rice.edu}


\title{High energy harmonic maps and degeneration of minimal surfaces}

\begin{abstract}

Let $S$ be a closed surface of genus $g \geq 2$ and let $\rho$ be a maximal $\text{PSL}(2, \mathbb{R}) \times \text{PSL}(2, \mathbb{R})$ surface group representation. By a result of Schoen, there is a unique $\rho$-equivariant minimal surface $\widetilde{\Sigma}$ in $\mathbb{H}^{2} \times \mathbb{H}^{2}$. We study the induced metrics on these minimal surfaces and prove the limits are precisely mixed structures. In the second half of the paper, we provide a geometric interpretation: the minimal surfaces $\widetilde{\Sigma}$ degenerate to the core of a product of two $\mathbb{R}$-trees. As a consequence, we obtain a compactification of the space of maximal representations of $\pi_{1}(S)$ into $\text{PSL}(2, \mathbb{R}) \times \text{PSL}(2, \mathbb{R})$.


\end{abstract}

\maketitle
\setcounter{tocdepth}{1}
\tableofcontents

\section{Introduction and main results}

Let $S$ be a closed, orientable, smooth surface of genus $g > 1$. For any reductive Lie group $G$, one can form the representation variety $\mathcal{R}(\pi_{1}(S), G) = \Hom^{+} (\pi_{1}(S), G) \sslash G$, consisting of conjugacy classes of reductive surface group representations into $G$. In the classical setting where $G = \text{PSL}(2, \mathbb{R})$, one recovers a copy of Teichm{\"u}ller space. A goal in the higher Teichm{\"u}ller theory is to understand geometric aspects of surface group representations into higher rank Lie groups. 


Following the work of Labourie \cite{Labourie08}, given a reductive surface group representation $\rho$ into a semi-simple Lie group $G$, to each complex structure $J$ on the surface $S$, one can record the energy of the unique $\rho$-equivariant harmonic map from $(\tilde{S}, J)$ to the Riemannian symmetric space $G/K$. This defines an energy functional on Teichm{\"u}ller space and Labourie proves that if the original representation $\rho$ is \emph{Anosov}, then the energy functional admits a critical point. Hence, to each such representation $\rho$, there is an associated branched immersed minimal surface in the symmetric space $G/K$.

The existence and uniqueness of the minimal surface in the associated symmetric space has been resolved by Labourie \cite{Labourie17} for the rank 2  real split simple Lie groups: namely SL$(3, \mathbb{R})$, PSp$(4, \mathbb{R})$ and $G_{2}$. Interestingly enough, the result still holds when $G$ is merely semi-simple, as the case of PSL$(2, \mathbb{R}) \times \text{PSL}(2, \mathbb{R})$ was proven by Schoen in \cite{Schoen}.

There is also the aim in the program of the higher Teichm{\"u}ller theory to understand representations as geometric objects. This is a natural goal, given that in the case of classical Teichm{\"u}ller theory, where the group $G$ = PSL$(2, \mathbb{R})$ and the representation is discrete and faithful, the associated geometric objects are given by marked hyperbolic surfaces. Moreover, it is of interest to obtain a description of boundary points associated to Hitchin components or spaces of maximal representations in terms of degenerations of geometric objects and to have these geometric objects at the boundary be a generalization of measured laminations (see \cite{Wien} section 11), which are the limiting geometric objects in the Thurston compactification of Teichm{\"u}ller space. 

In the setting $G=$ PSL(2, $\mathbb{R}) \times$ PSL$(2,\mathbb{R})$, this paper does exactly that-- we provide a parameterization of maximal surface group representations into PSL(2, $\mathbb{R}) \times$ PSL$(2,\mathbb{R})$, by studying the induced metrics on the $\rho$-equivariant minimal surfaces in the  symmetric space $\mathbb{H}^{2} \times \mathbb{H}^{2}$. If $\rho = (\rho_{1} , \rho_{2})$, and $\widetilde{\Sigma}$ is the unique $\rho$-equivariant minimal surface in $\mathbb{H}^{2} \times \mathbb{H}^{2}$, then its quotient by action of the fundamental group by the representation, is the graph of the unique minimal lagrangian isotopic to the identity between $(S, g_{1}) = \mathbb{H}^{2} / \rho_{1}$ and $(S, g_{2}) =\mathbb{H}^{2} / \rho_{2}$. 




Denote Ind$(S)$ to be the equivalence class of induced metrics on the graph minimal surface in the product of two hyperbolic surfaces. Two such metrics are identified if one is the pullback metric of the other by a diffeomorphism homotopic to the identity map.

We study the length spectrum of these induced metrics on the minimal surface and show that we can degenerate the metrics to obtain singular flat metrics, measured laminations and mixed structures. A \emph{mixed structure} $\eta = (S_{\alpha}, q_{\alpha}, \lambda)$ is the data of a collection of incompressible subsurfaces $\sqcup  \, S_{\alpha}$, with a prescribed meromorphic (integrable) quadratic differential on each subsurface (collapsing the boundary components and viewing them as punctures) and a singular flat metric on each subsurface coming from the prescribed quadratic differential, with a measured lamination $\lambda$ supported on the complement $S \, \backslash \sqcup \, S_{\alpha}$. Observe that a singular flat metric coming from a holomorphic quadratic differential on $(S,J)$ and a measured lamination on $S$ are trivial examples of mixed structures, where $S_{\alpha} = S$ and $S_{\alpha} = \emptyset$, respectively. Our first main result is the following.

\begin{thmx}
The space of induced metrics \emph{Ind($S$)} embeds into the space of projectivized currents \emph{PCurr($S$)}. Its closure is \emph{Ind($S$)} $\sqcup$ \emph{PMix($S$)}.

\end{thmx}




If we keep track of the ambient space, namely $\mathbb{H}^{2} \times \mathbb{H}^{2}$, we show that by scaling the ambient space by a suitable sequence of constants (which generally will be the total energy of a harmonic map), we can obtain as limits of minimal langrangians the core of a pair of $\mathbb{R}$-trees coming from measured foliations. In fact, we show there is an isometric embedding from a metric space obtained from the data of a mixed structure to the core of trees.

As a consequence, we have an answer to our original goal of ascribing something geometric to surface group representations to PSL$(2, \mathbb{R}) \times \text{PSL}(2, \mathbb{R})$ which are maximal, and a description of natural boundary objects which are geometric and are natural extensions of measured laminations.

\begin{thmx}
The space of maximal representations of \emph{PSL}$(2,\mathbb{R})$ $\times$ \emph{PSL}$(2.\mathbb{R})$ embeds into the space of $\pi_{1}S$-equivariant harmonic maps from $\mathbb{H}^{2} \to \mathbb{H}^{2} \times \mathbb{H}^{2}$, whose graphs are minimal lagrangians. The scaled Gromov-Hausdorff limits of these maps are given by harmonic maps from $\mathbb{H}^{2}$ to $T_{1} \times T_{2}$, where $T_{1}$ and $T_{2}$ are a pair of $\mathbb{R}$-trees coming from a projective pair of measured foliations, with image given by the core of the trees.

\end{thmx}




There has been some recent interest in studying surface group representations to the Lie group  PSL(2, $\mathbb{R}) \times$ PSL$(2, \mathbb{R})$ by way of geodesic currents. Work of Glorieux \cite{Glo17} shows that the average of two Liouville currents $\frac{L_{X_{1}} + L_{X_{2}}}{2}$ yields the length spectrum of the Globally Hyperbolic Maximal Compact AdS$^{3}$ manifold with holonomy $(\rho_{1}, \rho_{2})$, where $X_{i} = \mathbb{H}^{2} \, \backslash \, \rho_{i}$. In another recent paper of Glorieux \cite{Glo19}, it is shown that this map which sends unordered pairs of elements in Teichm{\"u}ller space to the space of projectivized currents given by $(X_{1}, X_{2})=(X_{2}, X_{1}) \to \frac{L_{X_{1}} + L_{X_{2}}}{2}$ is injective. Forthcoming work of Burger, Iozzi, Parreau, and Pozzetti \cite{BIPP} will show the limits of this embedding are given by the projectivization of a pair of measured laminations. The limiting currents thus satisfy 
\begin{align}
i( \eta, \cdot) = i(\lambda_{1}, \cdot) + i(\lambda_{2}, \cdot),
\end{align}
where $\lambda_{1}$ and $\lambda_{2}$ are specific representatives of the projectivize classes $[\lambda_{1}]$ and $[\lambda_{2}]$, respectively, representing limits on the Thurston boundary. 



We remark that our compactifcation via geodesic currents is distinct. If the limiting laminations $\lambda_{1}$ and $\lambda_{2}$ $\emph{fill}$, that is, the sum of their intersection numbers with any third measured lamination is never zero, then the corresponding limiting object $\eta'$ under our compactification is a singular flat metric coming from a unit-norm holomorphic quadratic differential $\Phi$ whose horizontal and vertical laminations are $\lambda_{1}$ and $\lambda_{2}$. The corresponding current is thus given by
\begin{align}
l^{2}_{|\Phi|}(\alpha) = i^{2}(\eta ', \alpha) = i^{2}(\lambda_{1}, \alpha) + i^{2}(\lambda_{2}, \alpha),
\end{align}
for a suitably short arc $\alpha$ away from the zeros of $|\Phi|$. In general, this is different from the sum of $\lambda_{1}$ and $\lambda_{2}$. Notice that for $\gamma$ an arc of the horizontal lamination of $\Phi$, then the two intersection numbers $i(\eta, \alpha)$ and $i(\eta', \alpha)$ coincide, so that the two currents $\eta$ and $\eta'$ are distinct even as projectivized currents.

\subsection*{Acknowledgments}
It is the author's privilege to thank his advisor Mike Wolf for first suggesting the problem and for his continued patience, support and guidance. The author would like to expresses his gratitude to Zheng (Zeno) Huang and Andrea Tamburelli for many fruitful conversations throughout this project.

\section{Geometric preliminaries}

\subsection{Harmonic maps between surfaces}

Let $(M, \sigma |dz|^{2})$ and $(N, \rho  |dw|^{2})$ be two closed Riemannian surfaces and $w:(M, \sigma |dz|^{2}) \to (N, \rho |dw|^{2})$ a Lipschitz map. Then the energy of the map $w$ is given by the integral
$$ \mathscr{E}(w) := \frac{1}{2} \int_{M} ||dw||^{2} \, dvol_{\sigma}.$$
A critical point of the energy functional is a {\it harmonic map}. We remark that if the domain $M$ is a surface, the energy is a conformal invariant; hence a harmonic map depends only upon the conformal class of the domain but depends on the metric of the target surface. 
The energy density of the map $w$ at a point is given by
$$e(w) = \frac{\rho(w(z))}{\sigma(z)} ( |w_{z}|^{2}  + |w_{\overline{z}}|^{2}),$$
and so the total energy is also given by the formula
\begin{align*} \mathscr{E}(w) &= \int_{M} e(w) \, \sigma \, dz d\overline{z}\\
&= \int_{M} \rho(w(z)) ( |w_{z}|^{2}  + |w_{\overline{z}}|^{2}) \, dz  d \overline{z},
\end{align*}
once again seeing that the total energy depends only upon the conformal structure of the domain and the metric of the target.
Alternatively, a harmonic map $w$ solves the Euler-Lagrange equation for the energy functional, a second-order nonlinear PDE:
$$w_{z \overline{z}} + (\log \rho)_{w} w_{z} w_{\overline{z}} = 0. $$

To any harmonic map $w:(M, \sigma |dz|^{2}) \to (N, \rho |dw|^{2})$, the  pull-back of the metric tensor decomposes by type according to
$$w^{*} \rho =  \Phi dz^{2} + \sigma e dz d \overline{z} + \overline{\Phi}d\overline{z}^{2},$$
where $\Phi \, dz^{2}$ is a holomorphic quadratic differential with respect to the complex structure coming from the conformal class of $(M, \sigma |dz|^{2})$ called the \emph{Hopf differential} of $w$. Much of the formulas arising from harmonic maps make use of the auxiliary functions:
\begin{align*}
\mathcal{H} &= \frac{\rho(w(z))}{\sigma(z)} |w_{z}|^{2}\\
\mathcal{L} &= \frac{\rho(w(z))}{\sigma(z)} |w_{\overline{z}}|^{2}.
\end{align*}
We list some of these formulas and make liberal use of them without always explicitly citing the precise one:
\begin{align*}
&\text{The energy density} \,\,\, e = \mathcal{H} + \mathcal{L}\\
&\text{The Jacobian} \,\,\, \mathcal{J} = \mathcal{H} - \mathcal{L}\\
&\text{The norm of the quadratic differential} \,\,\, |\Phi|^{2}/\sigma^{2} = \mathcal{HL}\\
&\text{The Lapace-Beltrami operator} \,\,\, \Delta \equiv \frac{4}{\sigma}\frac{\partial^{2}}{\partial z \partial \overline{z}} \\
&\text{Gaussian curvature of the source} \,\,\, K(\sigma) = - \frac{2}{\sigma} \frac{\partial^{2} \log \sigma}{\partial z \partial \overline{z}}  \\
&\text{Gaussian curvature of the target} \,\,\, K(\rho) = - \frac{2}{\rho} \frac{\partial^{2} \log \rho}{\partial w \partial \overline{w}} \\
& \text{The Beltrami differential} \,\,\, \nu = \frac{w_{\overline{z}}}{w_{z}} = \frac{\overline{\Phi}}{\sigma \mathcal{H}}\,\,\, \text{and} \,\,\,|\nu|^{2} = \frac{\mathcal{L}}{\mathcal{H}}.
\end{align*}
The Bochner formula is given by
\begin{align*}
\Delta \log \mathcal{H} &= -2K(\rho) \mathcal{H} + 2K(\rho) \mathcal{L} + 2 K(\sigma),  \qquad \text{when} \,\, \mathcal{H}(p) \neq 0\\
\Delta \log \mathcal{L} &= -2K(\rho) \mathcal{L} + 2K(\rho) \mathcal{H} + 2 K(\sigma),  \qquad \text{when} \,\, \mathcal{L}(p) \neq 0.
\end{align*}
We shall often be in the setting where both the source and target are hyperbolic surfaces, that is $K(\sigma) = K(\rho) \equiv -1$, and so some of the formulas listed above can be simplified. In the more general setting where the target has negative curvature, the existence of a harmonic map in the homotopy class is due to Eells-Sampson \cite{ES64}, its uniqueness is due to Hartman \cite{Har67} and Al'ber \cite{Al}, and that if the homotopy class contains a diffeomorphism, then the harmonic map itself is a diffeomorphism and $\mathcal{H} >0$ is due to Schoen-Yau \cite{SY78} and Sampson \cite{Sam78}.

\subsection{Teichm{\"u}ller space}
Recall that Teichm{\"u}ller space $\mathcal{T}(S)$ is the space of all hyperbolic metrics on $S$ with the identification  $g \sim h$, if there exists a diffeomorphism $\phi$ of the surface, homotopic to the identity map, for which $\phi^{*}g =h$. The topology is given by its marked length spectrum.

Alternatively, one may regard Teichm{\"u}ller space as the space of marked Riemann surfaces. For a fixed surface $S$, two complex structures $(S, J_{1})$ and $(S, J_{2})$ are identified if there exists a biholomorphism  $f: (S, J_{1}) \to (S, J_{2})$, which is homotopic to the identity. The topology is given by the metric which for two points of Teichm{\"u}ller space assigns the logarithm of the quasiconformal dilatation of the unique Teichm{\"u}ller mapping between the marked Riemann surfaces. 

Teichm{\"u}ller space is topologically trivial being homeomorphic to an open ball of dimension $6g-6$. 
  
\subsection{Measured foliations and measured laminations}
For a closed surface $S$, a measured foliation $(S, \mathcal{F})$ is a singular foliation with a transverse measure, that is a measure $\mu$ defined on each arc transverse to the foliation, such that the measure is invariant under isotopy between two arcs through transverse arcs. 

To any isotopy class of measured foliations, there is an associated measured lamination. A measured geodesic lamination on a hyperbolic surface is a closed disjoint set of geodesics with a transverse measure. Likewise, to any measured lamination, there is an associated measured foliation, so that there is a canonical way to pass from to the other (see \cite{BC}, \cite{HP}). Hence, the space of measured laminations does not depend upon the choice of hyperbolic metric. Thurston showed both spaces are homeomorphic to Euclidean balls of dimension $6g-6$ (see \cite{FLP}, \cite{Thurston}).

\subsection{Holomorphic quadratic differentials}

The space of holomorphic quadratic differentials $Q_{g}$ is a holomorphically trivializable vector bundle over Teichm{\"u}ller space, whose fiber over $X$ is the vector space of holomorphic quadratic differentials on $X$. It is the vector space of holomorphic sections of the square of the canonical bundle $K_{X}$, and so may be written $H^{0}(X, K_{X}^{2})$. By the Riemann-Roch theorem, the complex dimension of this vector space is $3g-3$. More concretely, if $X$ is a Riemann surface and $\Phi$ is a holomorphic quadratic differential on $X$, then locally $\Phi = f(z) dz^{2}$, where $f$ is a holomorphic function and $z$ is a chart for $X$. 

Holomorphicity of the differential and compactness of the Riemann surface ensures the quadratic differential has precisely $4g-4$ zeros counted with multiplicity. Hence, in a neighborhood avoiding a zero of $\Phi$, one may choose natural coordinates $\zeta$ so that $\Phi = d \zeta^{2}$. The metric $|\Phi|$ is well-defined on the complement of the zeros and is locally Euclidean. At the zeros, the metric has conic singularities of angle $(n+2) \pi$, where $n$ is the order of the zero of the quadratic differential at that point. 

For any point on the complement of the zeros of the quadratic differential, there is a unique direction for which $q(v,v) \in \mathbb{R}^{+}$. Integrating the resulting line field, one obtains a foliation, called the \emph{horizontal foliation} of the quadratic differential $q$. Likewise, one can define the \emph{vertical foliation} of $q$, by integrating the line field of directions for which $q(v,v) \in i \mathbb{R}^{+}$. The foliations come equipped with a transverse measure. For any arc $\gamma$ transverse to the horizontal foliation, the measure for the horizontal foliation is given by
$$\tau_{h} = \int_{\gamma} |\text{Im}(\sqrt{q})(z)| |dz|,$$ 
and likewise, the transverse measure for the vertical foliation is given by integrating the real part $|\text{Re}(\sqrt{q})|$ over and arc $\gamma$.

If $S_{g,n}$ is a compact surface of genus $g$ with $n$ punctures such that $3g-3+n > 0$, then $Q_{g,n}$ will denote the space of integrable holomorphic quadratic differentials on $S_{g,n}$. At each of the punctures, the differential has a pole of order 1.

\subsection{Geodesic Currents and Marked Length Spectra}

Let $(S, \sigma)$ be a fixed closed hyperbolic surface of genus $g\geq2$. Then its universal cover $\tilde{S}$ may be identified isometrically with $\mathbb{H}^{2}$. Let $G(\tilde{S})$ denote the space of geodesics of $\tilde{S}$. Then a geodesic current on $S$ is a $\pi_{1}(S)$-equivariant Radon measure on $G(\tilde{S})$. The space of geodesic currents, denoted Curr(S), is given by the weak$^{*}$ topology. 

\begin{rem} A priori, the definition of a geodesic current may appear to depend upon the choice of hyperbolic metric, but it turns out $G(\tilde{S})$ depends only upon $\pi_{1}(S)$ (c.f. \cite{Bon}), hence the space of geodesic currents is independent of the hyperbolic metric initially chosen for $S$.

\end{rem}

The ur-example of a geodesic current is given by a single closed geodesic $\gamma$ on $S$. Lift $\gamma $ to a discrete set of geodesics $\tilde{\gamma}$ on $\tilde{S}$. These lifted geodesics may be given a Dirac-measure, which is $\pi_{1}(S)$-invariant as the lifts themselves are $\pi_{1}(S)$-invariant. Hence to any closed curve, by looking at its geodesic representative, one obtains a geodesic current on $S$. In fact, Bonahon \cite{Bon} shows the space of weighted closed curves is dense in Curr$(S)$ and the geometric intersection number between curves has a continuous bilinear extension to $i:$ Curr$(S) \times$ Curr$(S) \to \mathbb{R}_{\geq0}$. Moreover, a geodesic current on $S$ is determined by its intersection number with all closed curves \cite{Otal}. The topology then on the space of geodesic currents is given by its marked length spectrum. For the fixed surface $S$, denote by $\mathscr{C}$ the set of isotopy classes of closed curves. The marked length spectrum of a geodesic current $\mu$ is given by the collection $\{ i(\mu, \gamma) \}_{\gamma \in \mathscr{C}}$. A sequence of geodesic currents $\mu_{n}$ is said to converge to $\mu$ if their marked length spectrums converge, that is, to each $\gamma \in \mathscr{C}$ and $\epsilon >0$, there is an $N(\epsilon, \gamma)$ so that for $n>N(\epsilon, \gamma)$, one has $|i(\mu, \gamma) - i(\mu_{n} , \gamma)| < \epsilon$. It is important to note that $N$ is allowed to depend on the curve class chosen. No requirement on uniform convergence is required.

If a current arises from a metric, the following rather useful formula applies:

\begin{prop}[Bonahon \cite{Bon}, Otal \cite{Otal}]\label{prop:area}
Let $\mu$ be a current arising from a metric $\sigma$. Then
$$i(\mu, \mu) = \frac{\pi}{2} \text{Area}(\sigma)$$
\end{prop}

In the case where $\mu$ is a geodesic current arising from a measured lamination, it is not hard to see that $i(\mu, \mu) =0$, but in fact, this turns out to be a characterization of measured laminations.

\begin{prop}[Bonahon \cite{Bon}]\label{ml}
Let $\mu$ be a geodesic current such that $i(\mu, \mu) =0$, then $\mu$ is a measured lamination.
\end{prop}

It is clear that if $\mu$ is a geodesic current, then so is $c \mu$ for $c \in \mathbb{R}_{+} $. The set of projectivized currents, denoted PCurr$(S)$ is given by Curr$(S)/ \sim$, where $\mu \sim \nu$ if there exists a positive constant $c$ for which $\mu = c \nu$ and so consists of projective classes of geodesic currents. The space PCurr$(S)$ is then given the quotient topology. We highlight an important property of this space.

\begin{prop}[Bonahon \cite{Bon}]
The space $\emph{PCurr}(S)$ is compact.
\end{prop}

Several geometric structures have been shown to be embedded into Curr$(S)$. The first such example was due to Bonahon \cite{Bon}, who showed Teichm{\"u}ller space could be embedded inside Curr$(S)$ via its Liouville current, namely $\sigma \mapsto L_{\sigma} $ with the property that for any closed curve $\gamma$, one has $l_{\sigma}([\gamma]) = i(L_{\sigma}, \gamma)$, so that the length of the geodesic representative of $\gamma$ with respect to the hyperbolic metric $\sigma$ coincides with the intersection number between the currents $L_{\sigma}$ and $\gamma$. As the space of measured laminations can be realized as geodesic currents, Bonahon recovers the Thurston compactification by way of projectivized geodesic currents.

Otal \cite{Otal} has shown the space of negatively curved Riemannian metrics on surfaces can be realized by geodesic currents. For any simple curve class $[\gamma]$, the length of the unique geodesic representative coincides with the intersection number of the corresponding geodesic current and the curve class $[\gamma]$, extending the work of Bonahon.

Duchin, Leininger and Rafi  \cite{DLR} have embedded the space of singular flat metrics arising from integrable holomorphic quadratic differentials. We summarize a few of results here, as we shall use them in what follows. Recall that to any holomorphic quadratic differential $q$, one can associate a singular flat metric $|q|$ via canonical coordinates.

The unit sphere $Q^{1}_{g} \subset Q_{g}$ consists of the holomorphic quadratic differentials with $L^{1}$-norm 1. Then the space Flat$(S)$ of unit-norm singular flat metrics may be identified by 
$$\text{Flat}(S) = Q^{1}_{g} \, / \, \mathbb{S}^{1},$$
where the action of $\mathbb{S}^{1}$ is given by multiplication by $e^{i \theta}$, for $0 \leq \theta \leq 2 \pi$. We require this quotient because if $q$ is a holomorphic quadratic differential, then $q$ and $e^{i \theta} q$ will have the same singular flat metric $|q|$. For $q \in Q_{g}^{1}$, consider the vertical foliation of $q$, that is $v_{q} = |\text{Re}(\sqrt{q})|$. Denote $v^{\theta}_{q} =  |\text{Re}(e^{i \theta} \sqrt{q})|$, the vertical foliation of $e^{i \theta} q$. Form the integral
$$ L_{q} := \frac{1}{2} \int_{0}^{\pi} v_{q}^{\theta} \, d \theta.$$

\begin{thm}[Duchin-Leininger-Rafi \cite{DLR}]\label{thm:DLR1}
The integral $L_{q}$ is a geodesic current such that to any simple closed curve $\gamma$,
$$l_{|q|}(\gamma) = i(L_{q}, \gamma),$$
where $|q|$ is the singular flat metric arising from the holomorphic quadratic differential $q$. Furthermore, the map which sends $|q| \in$ \emph{Flat(}$S\emph{)}$ to $L_{q} \in$ \emph{PCurr}$(S)$ is an embedding. 
\end{thm}

As the space of projectivized currents is compact, one may take the closure of the space Flat$(S)$, and it is shown \cite{DLR} that the limiting structures consist precisely of projectivized mixed structure. A \emph{mixed structure} may be defined as follows. Let $W$ be an incompressible subsurface of $S$. Then consider $Q_{W}$, the space of integrable holomorphic quadratic differentials on $W$, where we have chosen a complex structure on the smooth surface $W$ such that neighborhoods of boundary components of $\partial W$ are conformally punctured disks. To any such quadratic differential $q$, the corresponding singular flat metric on $W$ thus assigns length zero to any peripheral curve. Let $\lambda$ be a measured lamination supported on the complement $S \,  \backslash W$. The triple $(W, q, \lambda)$ is called a mixed structure on $S$. To any $\eta =(W, q, \lambda)$, one obtains a geodesic current $L_{\eta}$ given by the property
$$i(L_{\eta}, \gamma) = i(\lambda, \gamma) + \frac{1}{2} \int_{0}^{\pi/2} i(v_{q}^{\theta}, \gamma) \, d \theta,$$
where $\lambda$ is a closed curve on $S$. We remark that in the case $W = \emptyset$, then $\eta$ is a measured lamination on $S$, so that the space Mix$(S)$ properly contains ML$(S)$. The compactification of the singular flat metrics arising from unit-norm quadratic differentials is then given by the following theorem.

\begin{thm}[Duchin-Leininger-Rafi \cite{DLR}]\label{thm:DLR2}
The closure of \emph{Flat(}S\emph{)} in \emph{PCurr(}S\emph{)} is given by \emph{Flat(}S\emph{)} $\sqcup$  \emph{PMix(}S\emph{)}.
\end{thm}

\subsection{Anti-de Sitter space}
We are primarily concerned with the anti-de Sitter space of signature $(2,1)$, which is given by the quasi-sphere $x_{1}^{2} + x_{2}^{2} +x_{3}^{2} - x_{4}^{2} = -1$ inside $\mathbb{R}^{(2,2)}$ with the metric  $ds^{2} = dx_{1}^{2} + dx_{2}^{2} - dx_{3}^{2} - dx_{4}^{2}$. More precisely, 
$$\widehat{AdS^{3}} = \{ x \in \mathbb{R}^{(2,2)} :  \langle x,x \rangle = -1\}.$$

As the manifold is pseudo-Riemannian, tangent vectors $v \in T \widehat{AdS^{3}}$   come in one of the following three types:

\begin{enumerate}
\item[] Timelike if $\langle v,v \rangle <0$
\item[] Lightlike if $\langle v,v \rangle = 0$
\item[] Spacelike if $ \langle v,v \rangle >0.$
\end{enumerate}

The anti-de Sitter space $AdS^{3}$ is given by the projectivization of $\widehat{AdS^{3}}$, its double cover. The isometry group of $AdS^{3}$ is $\text{PSL}(2, \mathbb{R}) \times \text{PSL}(2, \mathbb{R})$.

A smooth surface $S \hookrightarrow \text{AdS}^{3}$ is said to be \emph{spacelike} if the restriction to $S$ of the metric on AdS$^{3}$ is a Riemannian metric. This is equivalent to the condition that every tangent vector $v \in TS$ is spacelike. 

Consider the Levi-Civita connections on $S$ and AdS$^{3}$ given by $\nabla$ and $\nabla^{S}$, respectively. For a unit normal field $N$ on $S$, the second fundamental form is given by
$$\nabla_{\tilde{v}}{\tilde{w}} = \nabla^{S}_{v} w + \text{II}(v,w)N.$$ 
The shape operator is the $(1,1)$ tensor given by $B(v) = \nabla_{v}N$. It satisfies the property II$(v,w) = \langle B(v), w \rangle$. The maximal surfaces then are governed by the condition that tr$B$=0.

An AdS$^{3}$ manifold is a Lorentzian manifold locally isometric to AdS$^{3}$. Among these manifolds, we restrict our attention to those which are``globally hyperbolic maximal compact", henceforth written as ``GHMC". These manifolds are defined by those satisfying the following three properties:

\begin{enumerate}
\item they contain a closed orientable space-like surface $S$
\item each complete time-like geodesic intersects $S$ precisely once
\item maximal with respect to isometric embeddings.
\end{enumerate}

It follows that GHMC AdS$^{3}$ manifolds must be homeomorphic to $S \times \mathbb{R}$. Mess \cite{Mess} showed that the genus of $S$ must be at least 2 and that GHMC structures are parametrized by two copies of Teichm{\"u}ller space. Barbot, B{\'e}guin and Zeghib \cite {BBZ} showed that for each such GHMC manifold, there exists a unique embedded spacelike maximal surface $\Sigma$. In fact, there is a parametrization of all such GHMC manifolds by the unique embedded maximal surface it contains along with its second fundamental form. 

\begin{thm}[Krasnov-Schlenker \cite{KS07}]
Let $M$ be a GHMC AdS$^{3}$-manifold and let $\Sigma$ be its unique embedded spacelike maximal surface. The second fundamental form of $\Sigma$ is given by the real part of a holomorphic quadratic differential on the underlying complex structure of the maximal surface. Furthermore, there is a homeomorphism between the space of all GHMC AdS$^{3}$-structures and the cotangent bundle of Teichm{\"u}ller space, which assigns to a GHMC AdS$^{3}$-structure, the conformal class of its unique maximal surface and the holomorphic quadratic differential for which its real part is the second fundamental form.
\end{thm}

The induced metric of the maximal surface is given by $e^{2u} \sigma$, where $\sigma$ is the hyperbolic metric and $u$ satisfies the following PDE:

$$\Delta_{\sigma} u = e^{2u} - e^{-2u} |\Phi| -1.$$

But the solution to this PDE is $u= \frac{1}{2} \log \mathcal{H}$ for which the PDE becomes the usual Bochner equation. Here $\mathcal{H}$ is the holomorphic energy density arising from harmonic maps between closed hyperbolic surfaces. Hence, the induced metric of the maximal surface is given by $\mathcal{H} \sigma$. As a corollary of our main result, we will describe the limiting length spectrum of any sequence of induced metrics of the maximal surface.

\section{Minimal lagrangians}

A diffeomorphism $\phi: (S, g_{1}) \to (S, g_{2})$ is said to be minimal if its graph $\Sigma \subset (S \times S, g_{1} \oplus g_{2})$ with the induced metric is a minimal surface. Observe that if $\phi$ is minimal then so is $\phi^{-1}$. If $\omega_{1}$ and $ \omega_{2}$ denote the area forms of $g_{1}, g_{2}$ respectively, and if in addition $\Sigma \subset (S \times S, \omega_{1} - \omega_{2})$ is a Lagrangian submanifold, then we say $\phi$ is a minimal lagrangian.

\begin{thm}[Schoen \cite{Schoen}]\label{thm:Schoen}
If $g_{1}$ and $g_{2}$ are hyperbolic metrics on $S$, then there is a unique minimal lagrangian map $\phi: (S, g_{1}) \to (S, g_{2})$ in the homotopy class of the identity.
\end{thm}

Let $\Sigma$ denote the graph minimal surface with the induced metric. Then its inclusion into the product $i: \Sigma \to (S \times S, g_{1} \oplus g_{2})$ is a conformal harmonic map. A conformal map to a product space is a product of harmonic maps whose Hopf differentials sum to zero. Hence to any pair of points in Teichm{\"u}ller space, one may record the data of both the conformal structure of the minimal surface along with one of the Hopf differentials. The harmonic maps parametrization of Teichm{\"u}ller space which we record below ensures the map is bijective. Sampson proved injectivity and continuity of the map, and Wolf showed the map was surjective and admits a continuous inverse.

\begin{thm}[Sampson \cite{Sam78}, Wolf \cite{W89}]\label{thm:wolfparam}
Let $(S, \sigma)$ be a fixed hyperbolic surface. To any point in Teichm{\"u}ller space $[(S, \rho)]$, select the representative $(S, \rho)$ so that the identity map $id: (S, \sigma) \to (S, \rho)$ is the unique harmonic map in its homotopy class and denote its Hopf differential $\Phi(\rho)$. Then this map
$$ \Phi: \mathcal{T}(S) \to H^{0}(X, K_{X}^{2})$$
is a homeomorphism, where $X$ is the complex structure associated to $(S, \sigma)$.
\end{thm}

\begin{thm}\label{thm:homeo}
The map 
\begin{align*}
\Psi: \mathcal{T}(S) \times \mathcal{T}(S) \to \ Q_{g}\\
(X_{1}, X_{2}) \mapsto \emph{Hopf} (u_{1}) \\
\end{align*}
which assigns to any pair of points $X_{1}, X_{2}$ in Teichm{\"u}ller space, the conformal structure of the unique graph minimal surface $\Sigma \subset X_{1} \times X_{2}$ along with the Hopf differential Hopf$(u_{1})$ of the projection $u_{1}: \Sigma \to X_{1}$ is a homeomorphism.
\end{thm}

\begin{proof}
The discussion above ensures the map $\Psi$ is well-defined. As the construction of the minimal surface varies continuously with the choice of $X_{1}, X_{2}$, it is clear the map is continuous. To see injectivity of $\Psi$, suppose that $\Psi(X_{1}, X_{2}) = \Psi(Y_{1}, Y_{2}) = (\Sigma, \Phi)$. Then the harmonic maps $u_{1}: \Sigma \to X_{1}$ and $v_{1}: \Sigma \to Y_{1}$ have the same Hopf differentials, so by the harmonic maps parameterization, $X_{1} = Y_{1}$. The same argument forces $X_{2} = Y_{2}$. Surjectivity follows similarly, as to any choice of Riemann surface $\Sigma = (S, J)$ and holomorphic quadratic differential $\Phi$, there exists a unique hyperbolic metric $X_{1}= (S, g_{1})$, so that the identity map $ id: \Sigma \to X_{1}$ is a harmonic map with Hopf differential $\Phi$. Similarly one can find an $X_{2}$ arising from the Hopf differential $-\Phi$. Hence $\Psi(X_{1}, X_{2}) = (\Sigma, \Phi)$ which gives surjectivity. The inverse is clearly continuous as given the data of a Riemann surface and a holomorphic quadratic differential, the pair of hyperbolic metrics may be written explicitly and vary continuously, which suffices for the proof.

\end{proof}

\section{Embedding of the induced metrics}

In this section we study the induced metric on the graph minimal surfaces. Recall that given a pair $(X_{1}, X_{2})$ of hyperbolic surfaces, Theorem \ref{thm:Schoen} produces a graph minimal surface $\Sigma$ in the 4-manifold $(S \times S, g_{1} \oplus g_{2})$, where $X_{i} =(S, g_{i})$. If $m: (S, g_{1}) \to (S, g_{2})$ is the unique minimal map isotopic to the identity, then $id: (S, g_{1}) \to (S, m^{*}g_{2})$ is the unique minimal map isotopic to the identity, which in this case is the identity map. The graph $\Sigma$ then, is the diagonal in $S \times S$ and there is a canonical diffeomorphism from the $S$ to the diagonal in $S \times S$. The induced metric on $\Sigma$ thus furnishes a metric $g$ on $S$ by the pullback of this diffeomorphism. Henceforth, when we say \emph{induced metric}, we refer to this metric $g$ on $S$, and will use $\Sigma$ to denote $(S,g)$. We consider these metrics up to pullback by a diffeomorphism isotopic to the identity, and call this subspace of metrics Ind$(S)$ and endowing it with the compact-open topology. The remainder of the section is devoted towards studying geometric properties of the minimal surfaces and showing Ind$(S)$ can be embedded into PCurr$(S)$.

\begin{prop}\label{prop:sum}
Let $X_{1} = (S, g_{1}), X_{2} = (S, g_{2})$ and $\Psi(X_{1}, X_{2}) = (\Sigma, \Phi).$ Then the induced metric on the minimal surface $\Sigma$ is given by $g_{1} + m^{*}g_{2}$. Consequently, the induced metric is given by twice the $(1,1)$ part of a hyperbolic metric when expressed in conformal coordinates. 
\end{prop}

\begin{proof}
As in the discussion above, we may choose a suitable hyperbolic metric $X_{2} =(S, g_{2})$ in the equivalence class of $[X_{2}]$ to ensure the unique minimal map isotopic to the identity is the identity map. Hence, the graph of the minimal map is the diagonal in $S \times S$, so that (after identifying the diagonal with $S$) the harmonic map from the minimal surface $\Sigma$ to $X_{i}$ is given by the identity map. The first result then follows by definition of the product metric. Notice that the hyperbolic metric $g_{1}$ may be written in conformal coordinates on $\Sigma$ as $\Phi dz^{2} + \sigma e dz d\overline{z} + \overline{\Phi} d \overline{z}^{2}$.  As the minimal surface $\Sigma$ is mapped conformally into the product $X_{1} \times X_{2}$ of hyperbolic surfaces, then one obtains a pair $u_{i}: \Sigma \to X_{i}$ of harmonic maps, whose Hopf differentials, Hopf$(u_{1})$ and Hopf$(u_{2})$, sum to zero. Hence $g_{2}$ may be written in conformal coordinates on $\Sigma$ as $-\Phi dz^{2} + \sigma e dz d\overline{z} - \overline{\Phi} d \overline{z}^{2}$, for $|\Phi| =|-\Phi|$, so by a result of Sampson (Proposition \ref{prop:Sampson}), the energy densities will coincide. As the induced metric is given by the sum, the induced metric has local expression $2 \sigma e dz d \overline{z}$.
\end{proof}

\begin{prop}\label{prop:neg}
The induced minimal surfaces have strictly negative sectional curvature.
\end{prop}

\begin{proof}

For any point $p \in \Sigma$, it is clear that $K_{p} \leq 0$, as $\Sigma$ is a minimal surface in a NPC space, so we wish to show that $K_{p}  \neq 0$. The proof is by contradiction. Let $\{e_{1}, e_{2}\}$ be an orthonormal basis of $N_{p} \Sigma$. Now consider the 2-plane spanned by eigenvectors $X$ and $Y$ of the second fundamental form $\mathrm{II}$. Then one has $\mathrm{II}(X,Y) = \sum_{j=1}^{2} \mathrm{II}_{j}(X,Y) e_{j}$. Then the mean curvatures of the immersion are given by 
\begin{align}
H_{1} & = \mathrm{II}_{1} (X, X) + \mathrm{II}_{1}(Y,Y) = 0 \\
H_{2} &= \mathrm{II}_{2}(X,X) + \mathrm{II}_{2}(Y,Y) = 0
\end{align}
Then the Gauss equation tells us that at $p$,
\begin{align}
0 = Rm(X,Y,Y,X) &= \widetilde{Rm}(X,Y,Y,X) - \langle \mathrm{II}(X,X), \mathrm{II}(Y,Y) \rangle + \langle \mathrm{II}(X,Y), \mathrm{II}(X,Y) \rangle  \\
&= \widetilde{Rm}(X,Y,Y,X) + \sum_{j=1}^{2} \mathrm{II}_{j}(X,X) \mathrm{II}_{j}(Y,Y) - \sum_{j=1}^{2} \mathrm{II}_{j}(X,Y)^{2},
\end{align}
and as $\mathbb{H}^{2} \times \mathbb{H}^{2}$ is NPC, from (4.1), (4.2) and (4.4), it follows $\mathrm{II} \equiv 0$ at $p$ and that $\widetilde{Rm}(X,Y,Y,X) = 0$ at p. As $T (\mathbb{H}^{2} \times \mathbb{H}^{2}) \cong  T \mathbb{H}^{2} \oplus T \mathbb{H}^{2} $, we may write $X = X_{1} \oplus X_{2}$ and $Y = Y_{1} \oplus Y_{2}$. A simple calculation shows: 
\begin{align*}
0 = \widetilde{Rm}(X,Y,Y,X) &= Rm_{1}(X_{1}, Y_{1}, Y_{1}, X_{1}) + Rm_{2}(X_{2}, Y_{2}, Y_{2}, X_{2})\\
&=\kappa (X_{1}, Y_{1})\left( |X_{1}|^{2} |Y_{1}|^{2} -  \langle X_{1}, Y_{1} \rangle^{2} \right) + \kappa(X_{2}, Y_{2})\left( |X_{2}|^{2} |Y_{2}|^{2} - \langle X_{2}, Y_{2} \rangle^{2} \right) \\
&= -1 \cdot \left( |X_{1}|^{2} |Y_{1}|^{2} - \langle X_{1}, Y_{1} \rangle^{2} \right) -1 \cdot \left( |X_{2}|^{2} |Y_{2}|^{2} - \langle X_{2}, Y_{2} \rangle^{2} \right),
\end{align*}
which by Cauchy-Schwarz implies that $X_{1}$ and $Y_{1}$ (and also $X_{2}$ and $Y_{2}$) are linearly dependent, so that the map $u_{1_{*}}$ drops rank, a contradiction, as our surface was a graph.
\end{proof}

\begin{prop}\label{prop:sff}
The second fundamental form is given by

\begin{align*}
\emph{II}(E_{1}, E_{1}) &= \frac{-\Re{\Phi}(\sigma e)_{y}-\sigma e (\Im{\Phi})_{x}+\Im{\Phi}(\sigma e)_{x}}{\sigma e \sqrt{2 \sigma e( \sigma^{2} e^{2} - 4|\Phi|^{2})}} JE_{1} \\
&+ \frac{\Im{\Phi} (\sigma e)_{y} - \sigma e(\Re{\Phi})_{x}+ \Re{\Phi} (\sigma e)_{x}}{\sigma e \sqrt{2 \sigma e(\sigma^{2} e^{2}-4|\Phi|^{2})}} JE_{2}\\
\emph{II}(E_{2}, E_{2}) &=\frac{\Re{\Phi}(\sigma e)_{y}+\sigma e (\Im{\Phi})_{x}-\Im{\Phi}(\sigma e)_{x}}{\sigma e \sqrt{2 \sigma e( \sigma^{2} e^{2} - 4|\Phi|^{2})}} JE_{1} \\
&+ \frac{-\Im{\Phi} (\sigma e)_{y} + \sigma e(\Re{\Phi})_{x}- \Re{\Phi} (\sigma e)_{x}}{\sigma e \sqrt{2 \sigma e(\sigma^{2} e^{2}-4|\Phi|^{2})}} JE_{2} \\
\emph{II}(E_{1}, E_{2}) &= \frac{\Im{\Phi}(\sigma e)_{y}-\sigma e(\Re{\Phi})_{x}+ \Re{\Phi} (\sigma e)_{x}}{\sigma e \sqrt{2 \sigma e(\sigma^{2} e^{2}-4|\Phi|^{2})}} JE_{1}\\
&+\frac{-\sigma e (\Re{\Phi})_{y} + \Re{\Phi} (\sigma e)_{y} - \Im{\Phi} (\sigma e)_{x}}{\sigma e \sqrt{2 \sigma e(\sigma^{2} e^{2}-4|\Phi|^{2})}} JE_{2}.
\end{align*}

\end{prop}

\begin{proof}
For a choice of complex coordinates $z=x+ iy$ on the minimal surface $\Sigma$,  then $\frac{\partial}{\partial x}$ and $\frac{\partial}{\partial y}$ is an orthogonal frame. Denote then $E_{1} = \frac{\partial}{\partial x} / |\frac{\partial}{\partial x}|_{\Sigma}$ and $E_{2} = \frac{\partial}{\partial y} / |\frac{\partial}{\partial y}|_{\Sigma}$. Let $J$ be the almost complex structure on the 4-manifold $X_{1} \times X_{2}$, then $J = J_{1} \oplus J_{2}$, where $J_{i}$ is the almost complex structure arising from $X_{i} = (S, g_{i})$. As $\Sigma \subset X_{1} \times X_{2}$ is a lagrangian submanifold, then $\{E_{1}, E_{2}, JE_{1}, JE_{2}\}$ forms an orthonormal basis of $T(X_{1} \times X_{2}) \cong TX_{1} \oplus TX_{2}$ in this neighborhood. The second fundamental form then is given by
\begin{align*}
\text{II}(X, Y) = \sum_{j=1}^{2} \widetilde{g} (\widetilde{\nabla}_{X} Y, JE_{j}) JE_{j},
\end{align*}
where $\widetilde{g} = g_{1} \oplus g_{2}$ and $\widetilde{\nabla} = \nabla_{1} \oplus \nabla_{2}$.
We first calculate $\text{II}(E_{1}, E_{1})$. As the minimal surface metric is given by $2\sigma e |dz|^{2} = 2\sigma e (dx^{2} + dy^{2})$, one has
$$ 2\sigma e (dx^{2} + dy^{2}) \left( \frac{\partial}{\partial x}, \frac{\partial}{\partial x}\right) = 2\sigma e = \norm{\frac{\partial}{\partial x}}^{2}_{\Sigma},$$
so that
$$E_{1} = \frac{\frac{\partial}{\partial x}}{\sqrt{2\sigma e}}.$$
Similarly $E_{2}$ is given by
$$E_{2} = \frac{\frac{\partial}{\partial y}}{\sqrt{2\sigma e}}.$$
To calculate $JE_{1}$, we project $E_{1}$ to each of its factors and apply the almost complex structure on each of its factors, namely, we find the vector which has the same length and forms angle $\pi/2$ with the projected factor using the hyperbolic metric. This is the complex structure arising from the conformal class of the metric. To find $J_{1} E_{1} = a \frac{\partial}{\partial x} + b \frac{\partial}{\partial y} $ for instance, we observe first the hyperbolic metric on $X_{1}$ is given by
$$ \rho_{1} = \Phi dz^{2} + \sigma e dz d \overline{z} + \overline{\Phi} d \overline{z}^{2} = (2 \Re{\Phi} + \sigma e) dx^{2} - 4 \Im{\Phi} dx dy +(-2 \Re{\Phi} + \sigma e) dy^{2} $$
Hence we want to solve $ a \neq 0$, $b >0$ for which
\begin{align}
g_{1}\left(a \frac{\partial}{\partial x} + b \frac{\partial}{\partial y}, E_{1}\right) &= 0 \\
g_{1}\left(a \frac{\partial}{\partial x} + b \frac{\partial}{\partial y}, a \frac{\partial}{\partial x} + b \frac{\partial}{\partial y}\right) = g_{1}(E_{1}, E_{1}) &=\frac{2 \Re{\Phi} + \sigma e}{2 \sigma e}.
\end{align}
Some basic algebra yields that $a = \frac{2 \Im{\Phi}}{\sqrt{(2\sigma e)((\sigma e)^{2} - 4 |\Phi|^{2}})}$ and $b= \frac{2 \Re{\Phi} + \sigma e}{\sqrt{(2\sigma e)((\sigma e)^{2} - 4 |\Phi|^{2}})}$, so that
$$J_{1}E_{1} = \frac{2 \Im{\Phi}}{\sqrt{(2\sigma e)((\sigma e)^{2} - 4 |\Phi|^{2}})}  \frac{\partial}{\partial x} +  \frac{2 \Re{\Phi} + \sigma e}{\sqrt{(2\sigma e)((\sigma e)^{2} - 4 |\Phi|^{2}})} \frac{\partial}{\partial y}.$$
 Notice that the denominator appearing in $b$ is positive, as $2|\Phi| < \sigma e$ when $\Phi \equiv 0$, in which case the minimal surface is a totally geodesic subsurface. Now $J_{2} E_{1}$ is found similarly, and is given by
$$J_{2}E_{1} = \frac{-2\Im{\Phi}}{\sqrt{2 \sigma e ((\sigma e)^{2} - 4|\Phi|^{2})}} \frac{\partial}{\partial x} + \frac{-2 \Re{\Phi} + \sigma e}{\sqrt{2 \sigma e ((\sigma e)^{2} - 4|\Phi|^{2})}} \frac{\partial}{\partial y} .$$
The tangent vector given by $\widetilde{\nabla}_{E_{1}}E_{1}$ splits as $\nabla^{1}_{E_{1}}E_{1} \oplus \nabla^{2}_{E_{1}}E_{1}$. The Christoffel symbols for $g_{1} $ and $g_{2}$ can be readily calculated.
\begin{align*}
\nabla^{1}_{E{1}}E_{1} &= \nabla^{1}_{\frac{\frac{\partial}{\partial x}}{\sqrt{2\sigma e}}}{\frac{\frac{\partial}{\partial x}}{\sqrt{2\sigma e}}} \\
&= \frac{1}{\sqrt{2 \sigma e}} \left(\frac{1}{\sqrt{2 \sigma e}} \nabla^{1}_{\frac{\partial}{\partial x}} \frac{\partial}{\partial x} +  \left(\frac{1}{\sqrt{2 \sigma e}}\right)_{x} \frac{\partial}{\partial x} \right) \\
&= \frac{1}{\sqrt{2 \sigma e}} \left(\frac{1}{\sqrt{2 \sigma e}} \left(^{1}\Gamma_{11}^{1} \frac{\partial}{\partial x}  + ^{1}\Gamma_{11}^{2} \frac{\partial}{\partial y} \right) + \left(\frac{1}{\sqrt{2 \sigma e}}\right)_{x} \frac{\partial}{\partial x} \right), \\
&=\left( \frac{1}{2 \sigma e}  ^{1}\Gamma_{11}^{1} + \frac{1}{\sqrt{2 \sigma e}} \left(\frac{1}{\sqrt{2 \sigma e}}\right)_{x} \right) \frac{\partial}{\partial x} +  \frac{1}{2 \sigma e} ^{2}\Gamma_{11}^{2} \frac{\partial}{\partial y}\\
\end{align*}
where $^{1}\Gamma_{11}^{1}$ and $^{1}\Gamma_{11}^{2}$ are the usual Christoffel symbols, where the extra superscript denotes these are the ones for the metric $g_{1}$. There are given explicitly by
$$^{1}\Gamma_{11}^{1} = \frac{1}{2} \left( \frac{-2 \Re{\Phi} + \sigma e}{\sigma^{2} e^{2} - 4|\Phi|^{2}} (2 \Re{\Phi} + \sigma e)_{x} + \frac{2 \Im{\Phi}}{\sigma^{2} e^{2} - 4|\Phi|^{2}} \left((-4 \Im{\Phi}_{x}) - (2 \Re{\Phi} + \sigma e)_{y}\right) \right)$$
$$^{2}\Gamma_{11}^{1} = \frac{1}{2} \left( \frac{2 \Im{\Phi}}{\sigma^{2} e^{2} - 4|\Phi|^{2}} (2 \Re{\Phi} + \sigma e)_{x} + \frac{2\Re{\Phi}+ \sigma e}{\sigma^{2} e^{2} - 4|\Phi|^{2}} \left((-4 \Im{\Phi}_{x}) - (2 \Re{\Phi} + \sigma e)_{y}\right) \right).$$
Similarly, the same can be done for the metric $g_{2}$ and using the formula for II, one gets II($E_{1}, E_{1})$. The same can be done for the rest.
\end{proof}

It would be curious to see under what conditions different points in $Q_{g}$ would yield the same induced metric. One might hope that the space of induced metrics would be homeomorphic to $Q_{g}$, but the following result of Sampson shows this is not possible:

\begin{prop}[Sampson]\label{prop:Sampson}
For a fixed closed hyperbolic surface $X=(S,\sigma)$, if $\Phi_{1}$ and $\Phi_{2}$ are two Hopf differentials on $X$ arising from harmonic maps from $X$ to closed hyperbolic surfaces of the same genus, such that the norms $|\Phi_{1}|$ and $|\Phi_{2}|$ coincide, then the energy densities coincide, that is $e_{1} = e_{2}$.
\end{prop}

Hence, if we select two elements of $Q_{g}$, say $(X, \Phi_{1})$ and $(X, \Phi_{2})$, where $|\Phi_{1}| = |\Phi_{2}|$, but $\Phi_{1} \neq \Phi_{2}$, then the corresponding energy densities are the same and hence the corresponding induced metrics are the same.

The following proposition is a converse to the result of Sampson and shows this is the only situation for which the corresponding induced metrics coincide.

\begin{lem}\label{lem:zeno}
On a fixed closed hyperbolic surface, we have $e_{1} = e_{2}$ if and only if $|\Phi_{1}| = |\Phi_{2}|$.
\end{lem}

\begin{proof}
That $|\Phi_{1}| = |\Phi_{2}|$ implies $e_{1} = e_{2}$ is due to Sampson. Now suppose $e_{1} = e_{2}$, then $\mathcal{H}_{1} + \mathcal{L}_{1} = \mathcal{H}_{2} + \mathcal{L}_{2}$, so the Bochner formula  $\mathcal{4} \log \mathcal{H}_{i} = 2\mathcal{H}_{i} - 2\mathcal{L}_{i} -2$ may be rewritten as $\mathcal{4} \log \mathcal{H}_{i} = 4\mathcal{H}_{i} - 2e_{i} -2$. Subtracting the two equations for $i=1,2$ yields 
$$\mathcal{4} \log \frac{\mathcal{H}_{1}}{\mathcal{H}_{2}}= 4 (\mathcal{H}_{1} - \mathcal{H}_{2}).$$Now $\mathcal{H}_{i}>0$, so that the quotient $\mathcal{H}_{1}/ \mathcal{H}_{2}$ attains its maximum on the surface, which we claim is $1$, for if the maximum of $\mathcal{H}_{1}/ \mathcal{H}_{2}$ is greater than 1, then at the maximum (which is also the maximum of  $\log \frac{\mathcal{H}_{1}}{\mathcal{H}_{2}}$)
\begin{align*}
0 \geq \mathcal{4}  \log \frac{\mathcal{H}_{1}}{\mathcal{H}_{2}} = 4 (\mathcal{H}_{1} - \mathcal{H}_{2}) 
=4 \mathcal{H}_{2}\, \left(\frac{\mathcal{H}_{1}}{\mathcal{H}_{2}} - 1\right) >0,
\end{align*}
a contradiction, so that $\frac{\mathcal{H}_{1}}{\mathcal{H}_{2}} \leq 1$ and symmetrically $\frac{\mathcal{H}_{2}}{\mathcal{H}_{1}} \leq 1$, hence $\mathcal{H}_{1} = \mathcal{H}_{2}$ and so $\mathcal{L}_{1} = \mathcal{L}_{2}$, by the assumption on the energy densities. From the formula $|\Phi|^{2}/\sigma^{2} = \mathcal{H} \mathcal{L}$, the conclusion follows.

\end{proof}

\begin{cor}
The space of induced metrics \emph{Ind}$(S)$ may be identified with $Q_{g}/\sim$, where $(X, \Phi_{1}) \sim (Y, \Phi_{2})$ if $X=Y$ and $|\Phi_{1}| = |\Phi_{2}|$.
\end{cor}

We conclude this section by proving the space Ind$(S)$ can be embedded into the space of currents and that the embedding remains injective after projectivization, thereby obtaining an embedding into projectivized currents.

\begin{prop}
The space $\emph{Ind(}S\emph{)}$ can be realized as geodesic currents. 
\end{prop}

\begin{proof}
From Proposition \ref{prop:neg}, the induced metrics have strictly negative curvature, so by Otal \cite{Otal}, there is a well-defined embedding $\mathcal{C}: \text{Ind}(S) \to \text{Curr}(S)$, from the space of induced metrics on $S$ to the space of geodesic currents, which sends $2 \sigma e \mapsto L_{2 \sigma e}$, so that if $\gamma$ is a closed curve, then $l_{2 \sigma e}([\gamma]) = i(L_{2 \sigma e}, \gamma)$.
\end{proof}

The following lemma is a statement concerning energy densities and their failure to scale linearly.

\begin{lem}\label{lem:stupid}
On a fixed closed hyperbolic surface, if $e_{1} = ce_{2}$, then $c=1$, and hence $|\Phi_{1}| = |\Phi_{2}|$.
\end{lem}
\begin{proof}
Without loss of generality, suppose $c \geq 1$, else we may reindex so that $c \geq 1$. Then $\frac{\mathcal{H}_{1}}{\mathcal{H}_{2}} \leq c$, for if $\frac{\mathcal{H}_{1}}{\mathcal{H}_{2}} >c$, we locate the maximum of $\mathcal{H}_{1}/ \mathcal{H}_{2}$, and the Bochner formula at that point yields
\begin{align*}
0 \geq \mathcal{4} \log \frac{\mathcal{H}_{1}}{\mathcal{H}_{2}}  &= 4 (\mathcal{H}_{1} - \mathcal{H}_{2}) - 2(e_{1} -e_{2}) \\
&= 4 (\mathcal{H}_{1} - \mathcal{H}_{2}) -2(ce_{2}-e_{2})\\
&= 4 \mathcal{H}_{2} \, \left(\frac{\mathcal{H}_{1}}{\mathcal{H}_{2}} - 1\right) -2e_{2} \, (c-1)\\
& >4 \mathcal{H}_{2} \, (c-1) - 2e_{2} \, (c-1)\\
&= (c-1) (4\mathcal{H}_{2} - 2e_{2})\\
&= (c-1) (2\mathcal{H}_{2} - 2\mathcal{L}_{2}) = 2(c-1) \,\mathcal{J}_{2} >0,
\end{align*}
a contradiction. Notice the upper bound is actually attained, for at a zero of $|\Phi_{1}|$, we have that $\mathcal{L}_{1}$ vanishes and so at such a zero we have the equation
$$\mathcal{H}_{1} = c \mathcal{H}_{2} + c \mathcal{L}_{2},$$
and as we have $\mathcal{H}_{1}/\mathcal{H}_{2} \leq c$, it follows that $\mathcal{L}_{2}$ must also vanish whenever $\mathcal{L}_{1}$ does. In fact, we can say more about the zeros of $\mathcal{L}_{i}$. The condition on the energy densities yields the equality
$$0 = (c \mathcal{H}_{2} - \mathcal{H}_{1}) + (c \mathcal{L}_{2} - \mathcal{L}_{1}),$$
and the bound on the quotient $\mathcal{H}_{1}/\mathcal{H}_{2}$ implies that the first term is nonnegative so the second term is nonpositive, that is $c\mathcal{L}_{2} - \mathcal{L}_{1} \leq 0$ or $c \leq \mathcal{L}_{1}/ \mathcal{L}_{2}$ or $\mathcal{L}_{2}/ \mathcal{L}_{1} \leq 1/c$, so that the order of the zeros of $\mathcal{L}_{2}$ is greater than or equal to the order of zeros of $\mathcal{L}_{1}$. As $|\Phi| / \sigma^{2} = \mathcal{H} \mathcal{L}$ and $\mathcal{H} >0$, then both $\mathcal{L}_{1}$ and $\mathcal{L}_{2}$ have exactly $8g-8$ zeros counted with multiplicity, so that the order of vanishing of $\mathcal{L}_{1}$ is the same as that of $\mathcal{L}_{2}$ at every point of the surface. Hence the quadratic differentials $\Phi_{1}$ and $\Phi_{2}$ differ by a multiplicative constant $k \in \mathbb{C}$, that is $\Phi_{1} = k \Phi_{2}$. 
At the zero of $|\Phi_{2}|$ (and so also a zero of $|\Phi_{1}|$), which is a maximum of the quotient $\mathcal{H}_{1}/\mathcal{H}_{2}$, the Bochner equation now reads,


\begin{align*}
0 \geq \mathcal{4} \log \frac{\mathcal{H}_{1}}{\mathcal{H}_{2}} &= 2\mathcal{H}_{1} - \frac{2|\Phi_{1}|^{2}}{\sigma^{2} \mathcal{H}_{1}} - 2\mathcal{H}_{2} + \frac{2|\Phi_{2}|^{2}}{\sigma^{2} \mathcal{H}_{2}}\\
&=2(\mathcal{H}_{1} - \mathcal{H}_{2})\\
&=2 \mathcal{H}_{2} (c-1) \geq 0,
\end{align*}
which implies $c=1$, and by the previous lemma $|k|=1$.



\end{proof}

\begin{thm}\label{thm:embed}
The space of induced metrics $\textnormal{Ind}(S)$ embeds into $\textnormal{PCurr}(S)$.
\end{thm}

\begin{proof}
Let $\pi: \text{Curr}(S) \to \text{PCurr}(S)$ be the natural projection map. It suffices to show the map $\pi \circ \mathcal{C}: \text{Ind}(S) \to \text{PCurr}(S)$ is injective. If the image of two induced metrics under the map $\pi \circ \mathcal{C}$ coincide, that is $\sigma dz d \overline{z} = c \sigma' e' dz d \overline{z}$, where $c \in \mathbb{R}_{>0}$, then they will be in the same conformal class, so that $\sigma = \sigma'$. Then $e=ce'$, and by Lemma \ref{lem:stupid}, $c=1$.



\end{proof}

\begin{rem}
As the induced metrics are not scalar multiples of each other, we make a slight modification by dividing the induced metrics by 2 to ensure these metrics are now precisely the $(1,1)$-part of a hyperbolic metric when written in conformal coordinates rather than twice that. 

\end{rem}


\section{Compactification of the induced metrics}

In this section we identify the elements in the closure $\overline{\text{Ind}(S)}$ $\subset$ PCurr($S$). As the space of projectivized currents is compact, we obtain a compactification Ind($S$) $\sqcup$ PMix($S$) of the induced metrics from the embedding obtained in the previous section.

\subsection{Flat metrics as limits}

In a simple scenario where the conformal structure of the minimal surface remains fixed, we can describe the asymptotic behavior of the induced metric. We consider the simplest case where $X_{1,n}$ (and consequently $X_{2,n}$) lie along a harmonic maps ray, that is the sequence of Hopf differentials of the projection map onto the first factor is given by $t_{n} \Phi$, where $\Phi \neq 0$ and $t_{n} \to \infty$.

\begin{prop}\label{prop:easyflat}
Let $\sigma_{n}e_{n}$ be the induced metric where $\sigma_{n} = \sigma$ for all $n$, and the Hopf differentials of the harmonic maps $u_{1,n}: (S, \sigma) \to X_{1,n}$ are given by $t_{n} \Phi_{0}$, where $\Phi_{0}$ is a unit-norm quadratic differential on $(S, \sigma)$. Suppose $\mathcal{E}_{n} \to \infty$. Then everywhere away from the zeros of $|\Phi_{0}|$, one has
$$\lim_{n \to \infty} \frac{\sigma_{n}e_{n}}{\mathcal{E}_{n}} = |\Phi_{0}|.$$ 
\end{prop}
\begin{proof}
By construction, the Hopf differential of the harmonic map from $(S, \sigma)$ to $X_{1,n}$ is given by $t_{n}\Phi_{0}$, where $\Phi_{0}$ is a unit-norm quadratic differential. In a neighborhood away from any zero of $\Phi_{0}$, consider then the horizontal foliation of $\Phi_{n} = t_{n} \Phi_{0}$. By the estimates on the geodesic curvature of its image \cite{W91}, a horizontal arc of the foliation in this neighborhood will be mapped close to a geodesic in $X_{1n}$ (we do not reproduce the techniques here, as we will do so later in a slightly modified setting). Using normal coordinates for the target adapted to this geodesic and estimates on stretching \cite{W89}, we have that
$$(x,y) \mapsto (2t_{n}^{1/2}x, 0) +o(e^{-ct}),$$
where the constant $c$ only depends upon the domain Riemann surface, and the distance from the zero of the quadratic differential.
For the harmonic map from $(S, \sigma)$ to $X_{2,n}$, its Hopf differential is given by $-t_{n}\Phi_{0}$, so that an arc of its horizontal foliation, which is an arc of the vertical foliation of $t_{n}\Phi_{0}$, gets mapped close to a geodesic, yielding
$$(x,y) \mapsto (0, 2t_{n}^{1/2}y) +o(e^{-ct}).$$
Hence, as a map from $\Sigma$ to the 4-manifold $X_{1,n} \times X_{2,n}$ with the product metric, we have that
the induced metric $\sigma_{n}e_{n}$ in this neighborhood has the form $(4t_{n}+o(e^{-ct})) dx^{2} + 2o(e^{-ct})dxdy + (4t_{n} +o(e^{-ct})) dy^{2}$. Dividing by $4t_{n}$ and observing that for a high energy harmonic map, the total energy is comparable to twice the $L^{1}$-norm of the quadratic differential (Proposition \ref{prop:mikefirst}) and taking the limit yields the conclusion.

\end{proof}

\begin{prop}\label{prop:flat}

Suppose $\sigma_{n} e_{n}$ is a sequence of induced metrics such that $\sigma_{n} \to \sigma \in \mathcal{T}(S)$ and $\mathcal{E}_{n} \to \infty$, then after passing to a subsequence, there exists a sequence $t_{n}$ and a unit-norm quadratic differential $\Phi_{0}$ on $[\sigma]$ so that 
$$\lim_{n \to \infty} \frac{\sigma_{n} e_{n}}{\mathcal{E}_{n}} \to |\Phi_{0}|.$$ 
\end{prop}

\begin{proof}
The result follows from the compactness of unit-norm holomorphic quadratic differentials over a compact set in Teich$(S)$ and the argument in the previous proposition.

\end{proof}
As the previous results only show $C^{0}$ convergence in any neighborhood away from a zero of the quadratic differential, it is not quite so obvious we have convergence in the sense of length spectrum. The following technical proposition shows we actually do have convergence when the metrics are regarded as projectivized geodesic currents. With the length spectrum embedding (as given in Theorem \ref{thm:embed}), we now have sequences of points whose limits are the flat structures in the space of geodesic currents. 

\begin{prop}\label{prop:convex}
Let $\sigma_{n} e_{n}$ and $\mathcal{E}_{n}$ be in the same setting as above. Then as currents 
$$\frac{L_{\sigma_{n} e_{n}}}{\mathcal{E}_{n}^{1/2}} \to L_{|\Phi_{0}|}.$$
\end{prop}

\begin{proof}
As the topology of geodesic currents is determined by the intersection number against closed curves, it suffices to show that given any closed, non-null homotopic curve class $[\gamma]$ and $\epsilon >0$, there is an $N([\gamma], \epsilon)$ such that for $n > N$, one has that $|i(L_{\sigma_{n} e_{n}/\mathcal{E}_{n}}, \gamma) - i(L_{|\Phi_{0}|}, \gamma)| < \epsilon$. We choose a representative $\gamma$ of $[\gamma]$ to be a $|\Phi_{0}|$-geodesic with length $L = i(L_{|\Phi_{0}|}, \gamma)$ with some fixed orientation. As the estimate in Proposition \ref{prop:flat} does not hold near a zero $z_{i}$ of $|\Phi_{0}|$, the first step is to construct open balls $V_{i}$ of radius $\epsilon$ in the $|\Phi_{0}|$-metric about each zero $z_{i}$ of $\Phi_{0}$ (choosing $\epsilon$ sufficiently small) so that

\begin{enumerate}
\item[(i)] balls centered about distinct zeros do not intersect
\item[(ii)] if the curve $\gamma$ enters one of the neighborhoods $V_{i}$, then the curve $\gamma$ must intersect the zero $z_{i}$ before $\gamma$ exits $V_{i}$
\item[(iii)] $(1-\epsilon)\, C-(4g-2) \, \pi \, \epsilon > 0$, where $C$ is the systolic length of the surface $(S, |\Phi_{0}|)$.
\end{enumerate}

As $\Phi_{0}$ is holomorphic, the zeros are isolated, so we can easily ensure (i) is satisfied. If the curve $\gamma$ does not intersect $z_{i}$, then as $\gamma$ is a closed curve, the distance from $z_{i}$ to the curve $\gamma$ in the $|\Phi_{0}|$-metric is bounded away from zero, guaranteeing condition (ii). Finally, condition (iii) follows as the systolic length $C$ of $(S, |\Phi_{0}|)$ and the genus of surface are fixed.

As the complement of the union of the $V_{i}$'s forms a compact set, by Proposition \ref{prop:flat} we can find an $N$ so that for $n > N$ the metrics $\sigma_{n} e_{n}/\mathcal{E}_{n}$ and $|\Phi_{0}|$ differ by at most $\epsilon$. Now each time $\gamma$ enters $V_{i}$, say at $p$, then hits the zero $z_{i}$ and exits $V_{i}$ for the first time thereafter, say at $q$, we may replace that segment of $\gamma$ with a segment running along the boundary of $V_{i}$ connecting $p$ and $q$. Notice this does not change the homotopy class of $\gamma$. We make this alteration for each instance $\gamma$ enters a $V_{i}$ and denote the new curve by $\gamma '$. Observe that each time we make such an alteration, the length of the curve (in the $|\Phi_{0}|$ metric) increases by at most $K_{i} \epsilon$, where $K_{i}$ is a constant depending only upon the $|\Phi_{0}|$ and the order of the zero $z_{i}$. In fact $K_{i} \leq (4g-2) \pi$. Hence the $|\Phi_{0}|$-length of $\gamma '$ is bounded above by $L + \sum_{i=1}^{j} n_{i} K_{i} \epsilon$, where $n_{i}$ is the number of times $\gamma$ enters $V_{i}$. But as $\gamma '$ now lies in the complement of the union of the $V_{i}$'s, by Proposition \ref{prop:flat}, the length of $\gamma '$ in the $\sigma_{n} e_{n}/\mathcal{E}_{n}$ metric is at most $(1+\epsilon)(L +\sum_{i=1}^{j} n_{i} K_{i} \epsilon)$. But the length of $\gamma'$ in the $\sigma_{n} e_{n}/\mathcal{E}_{n}$ metric must be at least the length of the geodesic in its homotopy class, which has length $L'_{n} = i(L_{\sigma_{n} e_{n}/\mathcal{E}_{n}}, \gamma)$, hence
\begin{align*}
(1+\epsilon)(L +\sum_{i=1}^{j} n_{i} K_{i} \epsilon) &\geq L'_{n}.
\end{align*}
Distributing on the left hand side and subtracting both sides by $L$, yields
\begin{align*}
\sum_{i=1}^{j} n_{i} K_{i} \epsilon + \epsilon (L + \sum_{i=1}^{j} n_{i} K_{i} \epsilon) &\geq L'_{n} - L \\
&:=i(L_{\sigma_{n} e_{n}/\mathcal{E}_{n}}, \gamma) - i(L_{|\Phi_{0}|}, \gamma).
\end{align*}

Now if $L'_{n}-L \geq 0$, then we are done, for $K_{i}$ is independent of $\epsilon$ and $n_{i}$ decreases as $\epsilon$ does.

So consider the case where $L'_{n}- L < 0$, that is, $L> L'_{n}$. Consider the $\sigma_{n} e_{n}/\mathcal{E}_{n}$-geodesic $\widetilde{\gamma}_{n}$ in the homotopy class of $\gamma$, and again we give $\widetilde{\gamma}_{n}$ an orientation. Naturally $\widetilde{\gamma}_{n}$ can enter and exit the $V_{i}$'s multiple times, but we remark that as the distance function on a NPC space from a convex set is itself convex, then each time the curve leaves $V_{i}$, it must pick up some topology before returning, that is, the part of the curve rel endpoints on lying on the boundary of $V_{i}$, the curve is not homotopic to a segment along the boundary of $V_{i}$.

However, now if $\widetilde{\gamma}$ enters and exits $V_{i}$ say a total of $r$ times, we consider the pairs of entry and exit points ordered accordingly as $p_{1}, q_{1}, .... , p_{r}, q_{r}$ using the chosen orientation. Now look at the segment of $\widetilde{\gamma}_{n}$  between $p_{s}$ and $p_{s+1}$. If this is homotopic rel endpoints to a segment of the boundary of $V_{i}$, then we look at the segment of $\widetilde{\gamma}_{n}$ between $p_{s}$ and $p_{s+2}$ (using a cyclic ordering so that $r+1$ is identified with 1) and see if that segment is homotopic rel endpoints to a segment along the boundary of $V_{i}$. We repeat this until the segment of $\widetilde{\gamma}_{n}$ between $p_{s}$ and $p_{s'}$ is not homotopic rel endpoints to the boundary of $V_{i}$. Then we repeat this process for $p_{s}$ and $p_{s-1}$ (again using a cylic ordering) until we find the segment of $\widetilde{\gamma}_{n}$ between $p_{s}$ and $p_{s''}$ which is not homotopic rel endpoints to the boundary of $V_{i}$. Then we replace the segment of $\widetilde{\gamma}$ between $p_{s''+1}$ and $p_{s'-1}$ with a segment along the boundary of $V_{i}$ connecting these two points. We repeat this for each $i$, so that when the curve leaves $V_{i}$, it picks up some topology before reentering $V_{i}$. Altering $\widetilde{\gamma}_{n}$ in this fashion yields a curve $\widetilde{\gamma}'_{n}$ lying outside of all the $V_{i}$'s. Switching over the $|\Phi_{0}|$-metric yields the inequality,
$$(1+\epsilon)L'_{n} +\sum_{i=1}^{4g-4} m_{i} K_{i} \epsilon \geq L,$$
where $m_{i}$ is the number of segments of the altered curve $\widetilde{\gamma}'_{n}$ lying on the boundary $V_{i}$ and once again $K_{i}$ is a constant depending solely on the order of the zero $z_{i}$. By assumption that $L > L'_{n}$, we have actually that
\begin{align*}
L'_{n}+\epsilon L +\sum_{i=1}^{4g-4} m_{i} K_{i} \epsilon &\geq L \\
\epsilon L +\sum_{i=1}^{4g-4} m_{i} K_{i} \epsilon & \geq L - L'_{n}.
\end{align*}
It suffices to show that $m_{i}$ can be bounded independently of $n$. This follows from an estimate on the systolic length of the metric $\sigma e_{n}/\mathcal{E}_{n}$. Let $C'$ denote the systolic length among all homotopically non-trivial curves which avoid the $V_{i}$'s for the metric $|\Phi|$. Then $C' \geq C$. Then by Proposition \ref{prop:easyflat}, the systolic length among all homotopically non-trivial curves which avoid all the $V_{i}$'s for the metric $\sigma e_{n}/\mathcal{E}_{n}$ is at least $(1-\epsilon)C'$. 

If $K$ denotes the largest constant among the $K_{i}$'s, then one has that
$$ \sum_{i=1}^{4g-4} m_{i} \leq \frac{L}{(1-\epsilon)C -K \epsilon},$$
for by construction we had $m_{i}$ segments of $\widetilde{\gamma}'_{n}$ which are each not homotopic rel endpoints to the boundary of $V_{i}$, so that if we connect the endpoints of the segment with a segment along the boundary of $V_{i}$, we add at most $K\epsilon$ to the length of the segment. But we now have a closed curve not homotopic to the boundary of any of the $V_{i}$'s, so the length of this closed curve is bigger than $C'$. This suffices for the proof.

\end{proof}

The resulting flat metrics arising from unit-norm holomorphic quadratic differentials are distinct as Riemannian metrics from the induced metrics as the quadratic differential metrics have zero curvature away from the zeros, whereas the induced metrics have negative curvature everywhere (Proposition \ref{prop:neg}). In fact, the flat metrics are distinct as geodesic currents, as work of Frazier \cite{Frazier} shows the marked length spectrum distinguishes nonpositively curved Euclidean metrics from the negatively curved Riemannian metrics.


\subsection{Measured Laminations as limits}


However, not all limits of induced metrics are given by flat metrics. One can also obtain measured laminations. This is most readily seen in the setting where one takes a hyperbolic metric and looks at the minimal lagrangian to itself. The induced metric of the minimal surface is then twice the hyperbolic metric. We thus have a copy of Teichm{\"u}ller space inside the space of induced metrics inside the space of projectivized currents. From Bonahon \cite{Bon}, we know we must have projectivized measured laminations in our compactification of the induced metrics. However, there are more ways to obtain measured laminations than by degenerating only the induced metrics which are scalar multiples of hyperbolic metrics, as the following proposition shows. 

\begin{prop}
Suppose $L_{\sigma_{n} e_{n}}$ leaves all compact sets, but that the sequence $\mathcal{E}_{n}$ of total energies is bounded, then in $\textnormal{PCurr}(S)$, we have  $[L_{\sigma_{n} e_{n}}] \to [\lambda] \in \textnormal{PMF}(S)$. Furthermore, if $[L_{\sigma_{n}}] \to [\lambda '] $ in the Thurston compactification, then $i(\lambda, \lambda ') =0$, where $\lambda \in [\lambda]$ and $\lambda' \in [\lambda']$.
\end{prop}

\begin{proof}
By the compactness of PCurr$(S)$, any sequence $[L_{\sigma_{n} e_{n}}]$ subconverges to $[\lambda] \in $ PCurr$(S)$. Hence, there is a sequence of positive real numbers so that $t_{n}L_{\sigma_{n} e_{n}} \to \lambda \in $ Curr$(S)$. We claim $t_{n} \to 0$. 

Consider a finite set of curves $\gamma_{1}, \gamma_{2}, ..., \gamma_{k}$ which fill the surface $S$. Then the current $\gamma_{1} + \gamma_{2} + ... + \gamma_{k}$ is a \emph{binding current}, that is to say, it has positive intersection number with any non-zero geodesic current. 

As $L_{\sigma_{n} e_{n}}$ leaves all compact sets in Curr$(S)$, then 
$$\lim_{n \to \infty} i(L_{\sigma_{n} e_{n}}, \gamma_{1} + ... \gamma_{k}) \to \infty,$$ so by continuity of the intersection form, one has 

$$\lim_{n \to \infty} t_{n} i(L_{\sigma_{n} e_{n}}, \gamma_{1} + ... + \gamma_{k}) = i(\lambda, \gamma_{1} + ... + \gamma_{k}).$$

But the intersection number on the right hand side is finite, hence $t_{n} \to 0$. From Proposition \ref{prop:area}, one has $i(L_{\sigma_{n} e_{n}} , L_{\sigma_{n} e_{n}}) = \pi/2 $ Area$(S, \sigma_{n} e_{n})$, which in this case is $\frac{\pi}{2} \mathcal{E}_{n}$. Then 
\begin{align*}
i(\lambda, \lambda) &= \lim_{n \to \infty} i(t_{n} L_{\sigma_{n} e_{n}}, t_{n} L_{\sigma_{n} e_{n}}) \\
&= \lim_{n \to \infty} t_{n}^{2} \frac{\pi}{2} \mathcal{E}_{n} =0 ,
\end{align*}
 where the last equality follows from the boundedness of total energy, hence $\lambda \in $ MF($S$). Now if $[L_{\sigma_{n}}] \to [\lambda']$, then there is a sequence $t_{n}' \to 0$ such that $t_{n}' L_{\sigma_{n}} \to \lambda '$. Then
\begin{align*}
i(\lambda, \lambda') &= \lim_{n \to \infty} i(t_{n} L_{\sigma_{n} e_{n}} , t_{n}' L_{\sigma_{n}})\\
&\leq \lim_{n \to \infty} t_{n} t_{n}' i(L_{\sigma_{n} e_{n}}, L_{\sigma_{n} e_{n}}) \\
&= \lim_{n \to \infty} t_{n} t_{n}' \mathcal{E}_{n} =0.
\end{align*}
where the inequality follows from $\sigma_{n} \leq \sigma_{n} e_{n}$ as metrics, and the last equality by the boundedness of the sequence of total energy $\mathcal{E}_{n}$ along with the sequences $t_{n}, t_{n}'$ tending towards zero.
\end{proof}

\subsection{Mixed structures as limits}
As some of the possible limits are the singular flat metrics arising from a holomorphic quadratic differential, the closure of the space of induced metrics on the minimal surface must include mixed structures, as these arise as limits of singular flat metrics. The main theorem asserts these are precisely all the possible limits of the degenerating minimal surfaces.

\begin{thm}\label{thm:compact}
Let $\sigma_{n} e_{n}$ be a sequence of induced metrics such that either $\sigma_{n}$ leaves all compact sets in $\mathcal{T}(S)$ or $\mathcal{E}_{n} \to \infty$, then there exists a sequence $t_{n} \to 0$ so that up to a subsequence $t_{n}L_{\sigma_{n}e_{n}} \to \eta = (S', q, \lambda) \in \textnormal{Mix}(S) \subset \textnormal{Curr}(S)$. Furthermore, given any $\eta \in \emph{Mix($S$)}$, there exists a sequence of induced metrics $\sigma_{n} e_{n}$ and a sequence of constants $t_{n} \to 0$, so that $t_{n} L_{\sigma_{n} e_{n}} \to \eta$. Hence, the closure of the space of induced metrics in the space of projectivized currents is $\overline{\textnormal{Ind}(S)} = \textnormal{Ind}(S) \sqcup \textnormal{PMix}(S)$.
\end{thm}

The proof of the main theorem will follow from a series of intermediate results, and will be at the end of the section. The strategy is to show that if the sequence of currents coming from the induced metrics is not converging projectively to a measured lamination, then scaling the induced metrics to have total area 1 is enough to ensure convergence in length spectrum. To each normalized induced metric, we produce a quadratic differential metric in the same conformal class as the induced metric, which will serve as a lower bound. Convergence of the quadratic differential metric to a mixed structure will yield a decomposition of the surface into a flat part and a laminar part. On each flat part, we will prove the conformal factor between the normalized induced metric and the quadratic differential converges to 1 uniformly (away from finitely many points). An area argument will show the complement is laminar.

The following proposition allows us to analyze sequences of induced metrics which are not converging to projectivized measured laminations. If the sequence of induced metrics is not converging to a projectivized measured lamination, we may scale the current associated to the induced metric by the square root of its area (which is also the total energy). We remark that in the case where the limiting geodesic current is not a measured lamination, then scaling the induced metrics by total energy of the associated harmonic map is strong enough to ensure length-spectrum convergence, yet delicate enough to ensure the limiting length spectrum is not identically zero. This should be compared to the situation in \cite{W91} and \cite{DDW98} where one always scales the metric by the total energy.

\begin{prop}\label{prop:scale}
Suppose the conformal class of the minimal surface leaves all compact sets in $\mathcal{T}(S)$ and the sequence of total energy is unbounded, that is $\mathcal{E}_{n} \to \infty$. Then up to a subsequence, there exists a sequence $c_{n} \to \infty$ and a geodesic current $\mu$ such that $c_{n} L_{\sigma_{n} e_{n}} \to \mu$. If $\mu$ is a measured lamination, then $c_{n} = o(\mathcal{E}_{n}^{1/2})$. If $\mu$ is not a measured lamination,  then $c_{n} \asymp \mathscr{E}_{n}^{1/2}$.
\end{prop}

\begin{proof}
By Theorem \ref{thm:embed}, one has an embedding of the space of induced metrics into the space of projectivized geodesic currents, which is compact. Taking the closure implies the first result. If $[\mu]$ is the limiting projective geodesic current, then one can choose a fixed representative; call it $\mu$. 

If $\mu$ is a measured lamination, then dividing the current $L_{\sigma_{n} e_{n}}$ by $\mathcal{E}_{n}^{1/2}$ normalizes the current to have self-intersection number 1. Then as the measured laminations have self-intersection 0, the second result follows.

Suppose then $\mu$ is not a measured lamination. Then its self-intersection number is positive and finite. But
\begin{align*}
i(\mu, \mu) &= \lim_{n \to \infty} i(c_{n} L_{\sigma_{n}e_{n}}, c_{n} L_{\sigma_{n} e_{n}})\\
&= \lim_{n \to \infty} c_{n}^{2} \, i(L_{\sigma_{n}e_{n}}, L_{\sigma_{n} e_{n}})\\
&= \lim_{n \to \infty} c_{n}^{2}\, \frac{\pi}{2} \,\text{Area}(S, \sigma_{n}e_{n})\\
&=\lim_{n \to \infty} c_{n}^{2} \, \frac{\pi}{2} \int_{S} \sigma_{n} e_{n} \,d z_{n} \, d \overline{z}_{n}\\
&= \lim_{n \to \infty} c_{n}^{2} \, \frac{\pi}{2} \,  \mathscr{E}_{n},
\end{align*}
so that
$0 < \lim_{n \to \infty} c_{n}^{2} \, \mathcal{E}_{n} < \infty$, that is $c_{n} \asymp \mathscr{E}_{n}^{1/2}$, as desired.

\end{proof}

With this normalization, the self-intersection of the current will be $\pi/2$, that is to say we have scaled the induced metric to have total area 1. 

The following proposition shows the relation of the induced metric to the corresponding Hopf differential metric.

\begin{prop}\label{prop:mike}
Away from the zeros of $\Phi$, one has the following identity
$$ \sigma e = |\Phi|\left(\frac{1}{|\nu|} + |\nu|\right).$$
Consequently,
$$\sigma_{n} e_{n} \geq 2|\Phi_{n}|.$$
\end{prop}

\begin{proof}
This result follows immediately by manipulation of the formulae involving $\mathcal{H}$ and $\mathcal{L}$. One has
\begin{align*}
\sigma^{2} e^{2} &= \sigma^{2} ( \mathcal{H}^{2} + 2 \mathcal{H} \mathcal{L} + \mathcal{L}^{2})\\
&= \sigma^{2} \mathcal{H} \mathcal{L} \, \left( \frac{\mathcal{H}}{\mathcal{L}} + 2 + \frac{\mathcal{L}}{\mathcal{H}} \right)\\
&= |\Phi|^{2} \left(\frac{1}{|\nu|}^{2} + 2 + |\nu|^{2} \right).
\end{align*}
Taking a square root on both sides yields the result.
\end{proof}

For a given sequence $\sigma_{n} e_{n}$ we consider the associated smooth function $f_{n} = (\frac{1}{|\nu_{n}|} + |\nu_{n}|)$. To each $n$, there is only one such $f_{n}$ by Lemma \ref{lem:zeno}. 

The following proposition due to Wolf allows us to pass freely between the $L^{1}$-norm of a Hopf differential and the total energy of the corresponding harmonic map. The original proof was for a fixed Riemann surface as the domain, but the argument holds when the domain is allowed to change. For the ease of the reader, we have included the adapted proof.

\begin{prop}[\label{prop:mikefirst}\cite{W89}, Lemma 3.2]
For any Riemann surface $(S,J)$ and hyperbolic surface $(S, \sigma)$, if $\emph{id}: (S,J) \to (S, \sigma)$ is a harmonic map with Hopf differential $\Phi$ and total energy $\mathcal{E}$, then

$$ \mathcal{E} + 2\pi \chi(s) \leq 2 ||\Phi|| \leq \mathcal{E} - 2\pi \chi(S).$$
\end{prop}

\begin{proof}
As $\mathcal{H} - \mathcal{L} = \mathcal{J}$ and $\int \mathcal{J} \sigma dz d \overline{z} = -2\pi \chi$, we have
$$\int \mathcal{H} \, \sigma \,dz d \overline{z} + 2\pi \chi = \int \mathcal{L} \, \sigma  \, dz d \overline{z} = \int \Phi \,\nu \,dz d \overline{z},$$
as the integrands agree. But, recalling that $|\nu| <1$, we have
\begin{align*}
\int \Phi \,\nu \,dz d \overline{z} &\leq \int |\Phi| \,dz d \overline{z} \\
&= \int \mathcal{H} \, |\nu| \, \sigma \,dz d \overline{z}\\
&\leq \int \mathcal{H} \, \sigma \, dz d \overline{z} = \int \mathcal{L} \, \sigma \, dz d \overline{z} - 2\pi \chi.
\end{align*}
Adding the first integral and the last integral yields
$$\int e \sigma \, dz d \overline{z} + 2 \pi \chi \leq 2 \int |\Phi| \, dz d \overline{z} \leq \int e \, \sigma \, dz d \overline{z} - 2\pi \chi,$$
proving the proposition.
\end{proof}

\begin{cor}
If the sequence $\Phi_{0,n}$ of unit-norm quadratic differential metrics converges projectively to a  measured lamination, then so does the associated sequence $L_{\sigma_{n} e_{n}} / \mathcal{E}^{1/2}_{n} $ of geodesic currents.
\end{cor}

\begin{proof}
Suppose $L_{|\Phi_{0,n}|} \to [\lambda]$ in the space of projectivized currents. As $i( L_{|\Phi_{0,n}|} , L_{|\Phi_{0,n}|}) = \pi/2$, while $i(\lambda, \lambda) =0$, then there exists a sequence $t_{n} \to 0$, so that the length spectrum of $t_{n} |\Phi_{0,n}|$ converges to that of some $\lambda \in [\lambda]$. This is to say, there is a curve class $[\gamma]$, for which the length of its geodesic representative against the metric $|\Phi_{0,n}|$ is unbounded, so by Propositions \ref{prop:mike} and \ref{prop:mikefirst}, the sequence of lengths of the $[\gamma]$-geodesic against the metrics $\sigma_{n} e_{n} / \mathcal{E}_{n}$ is unbounded. Hence there is a sequence $s_{n} \to 0$ so that $s_{n} L_{\sigma_{n} e_{n}} / \mathcal{E}^{1/2}_{n}$ converges to a current $\mu$. But as the self-intersection of $L_{\sigma_{n}e_{n}}/\mathcal{E}_{n}$ is exactly 1, the intersection of $\mu$ with itself is zero, from which the result follows.
\end{proof}

The previous corollary allows us to exclude the case where the sequence of flat metrics tends towards a projectivized measured lamination, for in that case, we have that the sequence of induced metrics also tends towards a projectivized measured lamination. Hence, we need only consider the case where the sequence of flat metrics converges to a non-trivial mixed structure, say $\eta$. The data of $\eta$ gives us a subsurface $S'$ for which the restriction of $\eta$ is a flat metric arising from a quadratic differential. Here we consider $S'$ up to isotopy. 

The remainder of the section is devoted towards showing that if the sequence of unit-norm quadratic differential metrics converges to a mixed structure that is not entirely laminar, then so does the sequence of unit-area induced metrics. This will then complete the proof of Theorem \ref{thm:compact}.

We begin by record the following useful bound due to Minsky, for the function $\mathcal{G} = \log (1/ |\nu|)$.
\begin{prop}[\label{prop:yair}\cite{Minsky}, Lemma 3.2]
Let $p \in S$ be a point with a neighborhood $U$ such that $U$ contains no zeros of $\Phi$ and in the $|\Phi|$-metric is a round disk of radius $r$ centered on $p$. Then there is a bound
$$ \mathcal{G}(p) \leq \sinh^{-1}\left( \frac{|\chi(S)|}{r^{2}} \right). $$
\end{prop}

\begin{proof}
The PDE $\Delta \mathcal{G} = 2 \mathcal{J} >0$ shows that $\mathcal{G}$ is subharmonic in $U$. It suffices therefore to bound the average of $\mathcal{G}$ on $U$ in the $|\Phi|$-metric. Some algebra yields
$$\sinh \mathcal{G} = \frac{1}{2} \frac{\sigma}{|\Phi|} \mathcal{J}.$$
Using the concavity of $\sinh^{-1}$ on the positive real axis, we obtain
\begin{align*}
\mathcal{G} (p) &\leq |\Phi|\text{-Avg}_{U} (\mathcal{G})  &&\text{by subharmonicity of $\mathcal{G}$} \\
&= |\Phi|\text{-Avg}_{U}\left(\sinh^{-1} \frac{1}{2} \frac{\sigma}{|\Phi|} \right) \\
&\leq \sinh^{-1} \left( |\Phi|\text{-Avg}_{U} \left(\frac{1}{2} \frac{\sigma}{|\Phi|} \right)\right)  &&\text{by concavity of $\sinh^{-1}$}\\
&= \sinh^{-1} \left( \frac{1}{2 \pi r^{2}} \int_{U} \frac{\sigma}{|\Phi|} \, \mathcal{J} \, d A(|\Phi|) \right)\\
&\leq \sinh^{-1} \left(\frac{|\chi(S)|}{r^{2}} \right) &&\text{by Gauss-Bonnet}.
\end{align*}
\end{proof}

As we are in the setting where the sequence $L_{\Phi_{0,n}}$ of currents coming from unit-area holomorphic quadratic differential metrics converges to a non-trivial mixed structure $\eta= (S', \Phi_{\infty}, \lambda)$, we have that the restriction of the metric $|\Phi_{0,n}|$ to $S'$ converges to the metric $|\Phi_{\infty}|$. On this systole positive collection $S'$ of subsurfaces, we have the following proposition.

\begin{prop}\label{prop:ae}
Given $\epsilon, \epsilon' >0$, there exists $N=N(\epsilon, \epsilon')$ such that for $n > N$ 
$$m_{|\Phi_{0,n}|} ( \{ p \in S : \left( \frac{1}{|\nu_{n}|} + |\nu_{n}| \right)(p) \geq 2 + \epsilon' \} ) < \epsilon.$$
Consequently the limiting function $\frac{1}{|\nu|} + |\nu|$ is 2 almost everywhere with respect to the $|\Phi_{\infty}|$-metric.
\end{prop}

\begin{proof}
By Proposition \ref{prop:mike}, one has the equality
\begin{align*}
\frac{\sigma_{n} e_{n}}{\mathcal{E}_{n}} &= \frac{|\Phi_{n}|}{\mathcal{E}_{n}} \left( \frac{1}{|\nu_{n}|} + |\nu_{n}| \right) \\
&= \left( \frac{||\Phi_{n}||}{\mathcal{E}_{n}} \right) \frac{|\Phi_{n}|}{||\Phi_{n}||} \left( \frac{1}{|\nu_{n}|} + |\nu_{n}| \right).
\end{align*}
Defining
$$\frac{1-c_{n}}{2} : = \left( \frac{||\Phi_{n}||}{\mathcal{E}_{n}} \right),$$
one has $c_{n} \to 0$ by virtue of Proposition \ref{prop:mikefirst}. Observe that $c_{n} \geq 0$, as the function $\frac{1}{|\nu_{n}|} + |\nu_{n}| \geq 2$, the area of  $|\Phi_{0,n}| =  \frac{|\Phi_{n}|}{||\Phi_{n}||}$ is 1 and the area of the scaled metric $\frac{\sigma_{n} e_{n}}{\mathcal{E}_{n}}$ is also 1. If $m_{n}$ then denotes the $|\Phi_{0,n}|$-measure of the set of of points for which the function $\frac{1}{|\nu_{n}|} + |\nu_{n}|$ is at least $2+ \epsilon '$, then one has 
\begin{align*}
\int_{\{p:(\frac{1}{|\nu_{n}|} + |\nu_{n}|)(p) \geq 2+ \epsilon'\}} \left(\frac{1}{|\nu_{n}|} + |\nu_{n}|\right) \left(\frac{1-c_{n}}{2}\right) \, \, d A(|\Phi_{0,n}|) &\\
+ \int_{\{p:(\frac{1}{|\nu_{n}|} + |\nu_{n}|)(p) < 2+ \epsilon'\}}  \left(\frac{1}{|\nu_{n}|} + |\nu_{n}|\right) \left(\frac{1-c_{n}}{2}\right) \, \, d A(|\Phi_{0,n}|)  &= \int dA(\frac{\sigma_{n} e_{n}}{\mathcal{E}_{n}}) =1.
\end{align*}
The integrand in the first integral is at least $(2 + \epsilon ')\left( \frac{1- c_{n}}{2} \right)$, whereas the second integrand is at least $2\left( \frac{1- c_{n}}{2} \right)$. Multiplying these lower bounds with the measures of their respective sets yields

$$(2 + \epsilon ')\left( \frac{1- c_{n}}{2} \right)\, m_{n} + 2 \, \left(\frac{1- c_{n}}{2} \right) \,(1-m_{n}) \leq 1.$$
Some basic algebraic manipulation gives
\begin{align*}
m_{n} \left( (2+ \epsilon')\, \left(\frac{1- c_{n}}{2} \right) - 2\, \left( \frac{1- c_{n}}{2} \right) \right) & \leq c_{n} \\
m_{n} \left(  \frac{1- c_{n}}{2}  \right) \epsilon' &\leq c_{n}\\
m_{n} &\leq \frac{2 c_{n}}{(1- c_{n})(\epsilon ')},
\end{align*}
and as $\epsilon '$ is now fixed, one may find a sufficiently large $N$ to guarantee $m_{n} < \epsilon$. As the metric $|\Phi_{\infty}|$ has finite total area, convergence in measure of the sequence of functions   $\frac{1}{|\nu_{n}|} + |\nu_{n}|$ to the constant function 2, implies that up to a subsequence, one has convergence to the constant function 2 almost everywhere. 
\end{proof}

Sets of measure zero can be rather problematic if we wish to say something about length of curves. The following proposition shows that we actually have convergence off the zeros and poles of $|\Phi_{\infty}|$.

\begin{prop}\label{prop:beltrami}
Suppose $\mathcal{E}_{n} \to {\infty}$. Then up to a subsequence $  \left( \frac{1}{|\nu_{n}|} + |\nu_{n}| \right) \to 2$ everywhere  on $S'$ except at the zeros and poles of $|\Phi_{\infty}|$.
\end{prop}

\begin{proof}
Observe that the function  $\frac{1}{|\nu_{n}|} + |\nu_{n}|$ is not defined at the zeros of $|\Phi_{n}|$, but is well-defined everywhere else. Moreover, the auxiliary function $\mathcal{G} = \log \frac{1}{|\nu|}$ satisfies the partial differential equation
$$ \Delta \log \frac{1}{|\nu_{n}|} = 2 \mathcal{J}_{n} > 0, $$
so that the function $ \mathcal{G}$ and hence  $\frac{1}{|\nu_{n}|} + |\nu_{n}|$ never attains an interior maximum on the complement of the zeros. It follows that  $\frac{1}{|\nu_{n}|} + |\nu_{n}|$ is only unbounded in a neighborhood of a zero of a corresponding quadratic differential $\Phi_{n}$. The sequence of flat metrics $|\Phi_{0,n}|$ on $S'$ converges geometrically to $|\Phi_{\infty}|$, and so the zeros of $|\Phi_{n}|$ on $S'$ will converge to the zeros of $|\Phi_{\infty}|$. For any $\epsilon >0$, consider balls of radius $3 \epsilon$ about each zero of $|\Phi_{\infty}|$, choosing $\epsilon$ sufficiently small, so that balls about distinct zeros do not intersect. Call this collection $B$. Then for large $n$, balls of radius $\epsilon$ in the $|\Phi_{0,n}|$ metric about the zeros of $|\Phi_{n}|$ will be contained in $B$. For each boundary component of $S'$, which in the geometric limit is collapsed to a puncture, choose a geodesic curve with respect to the $|\Phi_{\infty}|$-metric, homotopic to the puncture and enclosing the puncture, of length $l_{\epsilon}> 3 \epsilon$ so that the $|\Phi_{\infty}|$-distance of each point of the curve to the puncture is at least $ 3 \epsilon$, possibly choosing a smaller $\epsilon$ until such a configuration is possible. This gives an annulus for each boundary component of $S'$. Call the collection of these annuli $A$.

For any point in the complement of both $A$ and $B$, for large $n$, the injectivity radius with respect to the $|\Phi_{0,n}|$-metric is at least $\epsilon$ and the distance to any of the zeros is at least $\epsilon$. Moreover, each point $p$ in the region satisfies the property that any $q \in B_{\epsilon/2}(p)$ has injectivity radius at least $\epsilon/2$ and distance at least $\epsilon/2$ to any zero or the boundary of the cylindrical region. Hence, by Proposition $\ref{prop:yair}$, the value of $\log (1/ |\nu_{n}|)$ is at most $M_{\epsilon/2}$, where the constant no longer depends on $n$, once $n$ is chosen sufficiently large. As the function $\log (1/ |\nu_{n}|)$ is subharmonic, by the mean-value property, one has at any point $p$ in this set
\begin{align*}
\log (1/ |\nu_{n}|)(p) &\leq \int_{B_{\epsilon/2}(p)} \log (1/ |\nu_{n}|) \,\, d A_{|\Phi_{0,n}|} \\
&\leq \left( |\Phi_{0,n}|\text{-Area}(B_{\epsilon/2}(p) \right) \epsilon ' + M_{\epsilon/2} \epsilon'',
\end{align*}
for $n$ large enough so that $\log (1 / |\nu_{n}|) < \epsilon '$ outside a set of measure at most $\epsilon''$ by Proposition \ref{prop:ae}. As the choice of $\epsilon$ was arbitrary, the conclusion follows.

\end{proof}


These collection of propositions prove the following result:
\begin{thm}\label{thm:geomflat}
Suppose $L_{|\Phi_{0,n}|}$ converges to a non-trivial mixed structure $\eta$. Then the corresponding metrics $\sigma_{n} e_{n} / \mathcal{E}_{n}$ as $\mathcal{E}_{n} \to \infty$, restricted to $S'$ converges geometrically to $|\Phi_{\infty}|$.
\end{thm}

\begin{proof}
Defining $A$ and $B$ as in the previous proof, on the region $S' \backslash (A \cup B)$, Proposition \ref{prop:beltrami} guarantees that we have uniform bounds on the sequence of functions $1/|\nu_{n}| + |\nu_{n}|$ whose limit was the constant function 2. Hence by Arzel{\'a}-Ascoli, up to a subsequence, we have uniform convergence on this region. Hence, by the same argument as that of Proposition \ref{prop:convex}, the length spectrum of the scaled induced metric on this domain converges to the limiting length spectrum of the sequence $|\Phi_{0,n}|$, which is $|\Phi_{\infty}|$.
\end{proof}

\begin{proof}[Proof of Theorem 5.5]
Recall that to any flat metric arising from a holomorphic quadratic differential, one can find a sequence of induced metrics so that the chosen flat metric is the limit in the space of geodesic currents (Proposition \ref{prop:flat}). Hence by Theorem \ref{thm:DLR2}, any mixed structure $\eta$ can be obtained by a sequence $L_{\sigma_{n}e_{n}}$ of currents coming from the induced metrics. On the other hand, to any sequence of induced metrics leaving all compact sets, then either it converges projectively to a measured lamination or it does not. If it does not, then the corresponding sequence of normalized Hopf differential metrics must converge to a mixed structure which is not purely laminar. The previous theorem thus ensures there is a nonempty collection of incompressible subsurfaces, $S'$, on which the limiting current is a flat metric. But on the complement of $S'$, the current $\mu$ restricts to a measured lamination (as on this complement the areas of the metric tend to zero), the proof of Theorem \ref{thm:compact} is complete.
\end{proof}

\subsection{Dimension of the boundary}
We end this section with a remark about the compactification of the induced metrics. Recall the dimension of the space of induced metrics (being homeomorphic to $\mathcal{Q}_{g}/ \mathbb{S}^{1}$) was $12g-13$. The dimension of the singular flat metrics can be readily seen to be of dimension $12g-14$. The actual mixed structures are stratified by the subsurfaces for which the mixed structure is a flat metric. A subsurface of lower complexity yields fewer free parameters in the choice of a flat structure, and the extra choices one gains for a measured lamination on the complementary subsurface is strictly less in our loss of choice for the flat structure. Hence the boundary of the compactification of the induced metrics via projectivized geodesic currents is of codimension 1.

\section{Analysis of the limits}

In this section, we wish to relate the mixed structures with cores of $\mathbb{R}$-trees arising from laminations. To this end, we elucidate the relation between the mixed structure and the pair of projective measured laminations obtained from the pair of degenerating hyperbolic surfaces.

\subsection{$\mathbb{R}$-trees} 
Here we recall some basic facts about $\mathbb{R}$-trees. An $\mathbb{R}$-tree $T$ is a metric space for which any two points are connected by a unique topological arc, and such that the arc is a geodesic. Equivalently, if $(X,d)$ is a metric space, for any pair of points $x, y \in X$, define the segment $[x,y] = \{ z \in X |\, d(x,y) = d(x,z) + d(z,y) \}$. Then an $\mathbb{R}$-tree is a real non-empty metric space $(T,d)$ satisfying the following:
\begin{enumerate}
\item[(i)] for all $x,y \in T$, the segment $[x,y]$ is isometric to a segment in $\mathbb{R}$.
\item[(ii)] the intersection of two segments with an endpoint in common is a segment
\item [(iii)] the union of two segments of $T$ whose intersection is a single point which is an endpoint of each is itself a segment.
\end{enumerate}

We say that a group $\Gamma$ acts on $T$ \emph{by isometry} if there is a group homomorphism $\theta: \Gamma \to$ Isom($T)$. The action is from the left. An action is said to be \emph{small} if the stabilizer of each arc does not contain a free group of rank 2. An action is said to be \emph{minimal} if no proper subtree is invariant under $\Gamma$.

A particularly important class of $\mathbb{R}$-trees comes from the leaf space of a lift of a measured foliation on a closed surface to its universal cover. Any measured foliation $\mathscr{F}$ on a closed surface of genus $g \geq 2$ may be lifted to a $\pi_{1}S$-equivariant measured foliation on its universal cover.  The leaf space can be made into a metric space, by letting the distance be induced from the intersection number. Notice this is an $\mathbb{R}$-tree with a $\Gamma= \pi_{1}S$ action by isometries.
Naturally, not all $\mathbb{R}$-trees with a $\pi_{1}S$ action arise from this construction. A theorem of Skora \cite{Sk90} shows that an $\mathbb{R}$-tree with a $\pi_{1}S$-action comes from a measured foliation if and only if the action is small and minimal. 
Alternatively, one may start with a measured lamination $(\lambda, \mu)$ on $S$ and lift it to a measured lamination $(\widetilde{\lambda}, \mu)$ on the universal cover. Then an $\mathbb{R}$-tree may be formed by taking the connected components of $\widetilde{S} \backslash \widetilde{\lambda}$ with edges between two vertices if the two components were adjacent (separated by a geodesic), and then metrically completing the distance induced by the intersection number. The $\mathbb{R}$-tree comes equipped with a $\pi_{1}S$-action, and is $\pi_{1}S$-equivariantly isometric to the $\mathbb{R}$-tree constructed from the corresponding measured foliation. In what follows, we will deal exclusively with $\mathbb{R}$-trees with a $\pi_{1}S$-action coming from the leaf space of the lift of a measured foliation.

\subsection{Relation of flat metrics to $\mathbb{R}$-trees}

We obtain a classification of the flat parts of the mixed structure arising from the data of the limits of the sequences $X_{1,n}$ and $X_{2,n}$. Let $S'$ be a connected subsurface for which the limiting mixed structure $\eta$ is a flat metric. For each $n$, denote by $S_{n}'$ the subsurface isotopic to $S'$ such that the boundary components are geodesics with respect to the induced metric $\sigma_{n} e_{n} / \mathcal{E}_{n}$. Let $X_{1,n}'$ denote the restriction of the hyperbolic metric $X_{1,n}$ to the subsurface of $S$, in the same isotopy class of $S'$, but which has geodesic boundary with respect to the hyperbolic metric. Then let $u'_{i,n}$ denote the restriction to $S'_{n}$ of the harmonic map $u_{i,n}: (S, \sigma_{n}e_{n}) \to X_{i,n}$.

\begin{thm}\label{thm:flatanalysis}
Consider a connected component of $S'$. The sequence of harmonic maps $u'_{1,n}: (S'_{n}, \sigma_{n}e_{n}/\mathcal{E}_{n}) \to X_{1,n}/2\mathcal{E}_{n}$ converges to a $\pi_{1}(S')$-equivariant harmonic map $u': (S', |\Phi_{\infty}|) \to T_{1}$, where $T_{1}$ is the $\mathbb{R}$-tree dual to $\lambda_{1} = \lim_{n \to \infty} X_{1,n} / 2 \mathcal{E}_{n}$. The Hopf differential is given by $\Phi_{\infty}$. Likewise the same holds for $\lambda_{2}$ and $- \Phi_{\infty}$. Hence, the laminations are the vertical and horizontal foliations of $\Phi_{\infty}$.
\end{thm}

\begin{proof}
We begin by showing that $\lambda_{1}$ is a well-defined measured lamination in the projective class of $[\lambda_{1}]$, which is the limit on the Thurston boundary of the sequence $X_{1,n}$. This will follow from standard estimates on stretching and geodesic curvature of an arc of the horizontal foliation which avoids the zeros. This will be an adaptation of the argument employed in \cite{W91}, for the case where the domain conformal structure is fixed and the Hopf differentials lie along a ray.

We first show boundedness of the Jacobian. For any neighborhood $U$ of the surface which avoids a zero of $\Phi_{0,n}$ one has the usual PDE
\begin{align}
\Delta_{\sigma_{n}} \log \frac{1}{|\nu_{n}|^{2}} = 4 \mathcal{J}_{n} > 0,
\end{align}
and consequently,
\begin{align}
\Delta_{\sigma_{n}} ||\Phi_{n}|| \log \frac{1}{|\nu_{n}|^{2}} = 4 ||\Phi_{n}|| \mathcal{J}_{n} > 0.
\end{align}
Using the conformal invariance of harmonic maps, we replace the metric $\sigma_{n}$ on the neighborhood $U$ with a metric $\sigma'_{n}$ in the same conformal class as $\sigma_{n}$, but one which is flat on $U$. Subharmonicity of the function $||\Phi_{n}|| \log 1/|\nu_{n}|^{2}$ yields
\begin{align} 
||\Phi_{n}|| \log \frac{1}{|\nu_{n}|^{2}} (p) \leq \frac{1}{\pi R^{2}} \int_{B_{R}(p)}  ||\Phi_{n}|| \log \frac{1}{|\nu_{n}|^{2}} \,\, dA(\sigma '_{n})
\end{align}

on a ball of $\sigma'_{n}$ radius $R$ contained in $U$. Some algebra yields
\begin{align}
\mathcal{J}_{n}(p) \frac{||\Phi_{n}|| \log \frac{1}{|\nu_{n}|^{2}} (p)}{\mathcal{J}_{n}(p)} \leq \frac{1}{\pi R^{2}} \int_{B_{R}(p)} \frac{\mathcal{J}_{n}}{\mathcal{J}_{n}} ||\Phi_{n}|| \log \frac{1}{|\nu_{n}|^{2}} \,\, dA(\sigma '_{n}), 
\end{align}
and hence
\begin{align}
\mathcal{J}_{n}(p) \leq \frac{\mathcal{J}_{n}(p)}{||\Phi_{n}|| \log |\nu_{n}|^{-2}} \left( \sup_{q \in B_{R}(p)} \frac{||\Phi_{n}|| \log |\nu_{n}|^{-2}(q)}{\mathcal{J}_{n}(q)} \right) \frac{1}{\pi R^{2}} \int_{B_{R}(p)} \mathcal{J}_{n} \,\, d A(\sigma '_{n}).
\end{align}

But one has that
\begin{align}
\frac{\mathcal{J}_{n}}{||\Phi_{n}|| \log |\nu_{n}|^{-2}} = \frac{|\Phi_{0,n}|}{\sigma_{n} |\nu_{n}|} \frac{(1-|\nu_{n}|^{2})}{\log |\nu_{n}|^{-2}},
\end{align}
so that in applying Proposition \ref{prop:beltrami} to the expression (6.6), one obtains that (6.5) may be rewritten as
\begin{align}
\mathcal{J}_{n}(p) \leq c_{n} \int_{B_{R}(p)} \mathcal{J}_{n} \,\, dA(\sigma'_{n}),
\end{align}
where $c_{n}$ will depend on the metric $|\Phi_{0,n}|, |\nu_{n}|, R$  and $\sigma_{n}$. But on the neighborhood $U$, we know for sufficiently large $n$, we have that $|\Phi_{n}| \to |\Phi_{\infty}|$, and $|\nu_{n}| \to 1$ and $\sigma_{n} \to \sigma_{\infty}$, where $\sigma_{\infty}$ is the uniformizing metric of $\Phi_{\infty}$. Hence $c_{n}$ remains bounded on $U$.
But finally,
\begin{align}
\int_{B_{R}(p)} \mathcal{J}_{n} \,\, dA(\sigma'_{n}) = \int_{B_{R}(p)} \frac{\sigma'_{n}}{\sigma_{n}} \mathcal{J}_{n} \,\, dA(\sigma_{n}) \leq \sup_{U} \frac{\sigma'_{n}}{\sigma_{n}} \int_{M} \mathcal{J}_{n} \,\, dA(\sigma_{n}) \leq -2 \pi \chi(S) c'_{n},
\end{align}
where here $c'_{n}$ will only depend upon the injectivity radius of the metric $\sigma_{n}$ on the neighborhood $U$, which for large $n$ will be close to the injectivity radius of $\sigma_{\infty}$.

From (6.7), (6.8) and the PDE in (6.1), one obtains by elliptic regularity (see \cite{GT}, Problem 4.8a)  that $|\nu_{n}| \to 1$ in $C^{1, \alpha}(U)$, where $U$ does not contain a zero or pole of $\Phi_{\infty}$. 

In the natural coordinates of the quadratic differential, the hyperbolic metric $g_{1,n}$ is given by $(\sigma_{n} e_{n} + 2 ||\Phi_{n}||) d \zeta_{n}^{2} + (\sigma_{n} e_{n} - 2 ||\Phi_{n}||) d \eta_{n}^{2}$. 

Recall that the geodesic curvature of an arc of the horizontal foliation of $\Phi_{0,n}$ in the natural coordinates for $\Phi_{0,n} =d \zeta_{n}^{2} = d \xi_{n}^{2} + d \eta_{n}^{2}$ is given by the equation
\begin{align}
\kappa(\gamma)_{\eta = constant} = -\frac{1}{2g_{11} \sqrt{g_{22}}} \frac{\partial g_{11}}{\partial \eta_{n}}, \end{align}
so that for $\gamma$ an arc of the horizontal foliation of $\Phi_{0,n}$, one has
\begin{align}
\kappa(\gamma)_{\eta = \text{constant}} &=- \frac{1}{2(\sigma_{n} e_{n} + 2||\Phi_{n}||)(\sigma_{n} e_{n} - 2||\Phi_{n}||)^{1/2}} \frac{\partial}{\partial \eta_{n}} (\sigma_{n} e_{n} + 2||\Phi_{n}||) \\
&= - \frac{1}{2 \mathcal{J}_{n} (\sigma_{n} e_{n} + 2 ||\Phi_{n}||)^{1/2}} \frac{\partial}{\partial \eta_{n}} \sigma_{n} e_{n}.
\end{align}
But simple algebra yields that $\sigma_{n} e_{n} = ||\Phi_{n}|| |\Phi_{0,n}|(|\nu_{n}|^{-1} + |\nu_{n}|)$, so that in the natural coordinates as $|\Phi_{0,n}| \equiv 1$, one actually has $\sigma_{n} e_{n} = ||\Phi_{n}|| (|\nu_{n}|^{-1} + |\nu_{n}|)$. Hence
\begin{align}
\kappa(\gamma) & = \frac{1}{2} ||\Phi_{n}|| (1-|\nu_{n}|^{2}) \mathcal{J}_{n}^{-1} |\nu_{n}|^{-2} (\sigma_{n} e_{n} + 2||\Phi_{n}||) \frac{\partial}{\partial \eta_{n}}|\nu_{n}|\\
&=\frac{1}{2} ||\Phi_{n}|| \mathcal{H}_{n}^{-1} |\nu_{n}|^{-2}(\sigma_{n} e_{n} + 2||\Phi_{n}||)^{-1/2} \frac{\partial}{\partial \eta_{n}} |\nu_{n}|,
\end{align}
as $\mathcal{J}_{n} = \mathcal{H}_{n}(1-|\nu_{n}|^{2})$. As $||\Phi_{n}|| \mathcal{H}_{n}^{-1} = |\nu_{n}|/|\Phi_{0,n}|$, rewriting (6.13) gives
\begin{align}
\kappa(\gamma) = \frac{1}{2} \frac{1}{(|\Phi_{0,n}| \cdot |\nu_{n}|)^{1/2}} \cdot \frac{\partial}{\partial \eta_{n}} |\nu_{n}|,
\end{align}
and as $|\nu_{n}| \to 1$ in $C^{1,\alpha}(U)$, one obtains $\kappa_{\rho_{1,n}}(\gamma) = o(||\Phi_{n}||^{-1/2}) = o(\mathcal{E}_{n}^{-1/2})$.

Then to any arc $\gamma$ of the horizontal foliation of $\Phi_{n}$, one has that is is mapped close to its geodesic in the target hyperbolic surface. The following standard calculation on the stretching shows that by normalizing the target hyperbolic manifold by the total energy, the resulting length is given by the intersection number with the measured lamination $\lambda_{1}$.
One has
\begin{align*}
l_{\rho_{1n}}(\gamma) &= \int_{\gamma} \mathcal{H}_{n}^{1/2} + \mathcal{L}_{n}^{1/2} \,\,ds_{\sigma_{n}}\\
&=\int_{\gamma}\mathcal{H}_{n}^{1/2}(1+|\nu_{n}|) \,\, ds_{\sigma_{n}} \\
&=\int_{\gamma} \frac{||\Phi_{n}||^{1/2} |\Phi_{0}|^{1/2}}{|\nu_{n}|^{1/2}} (1+|\nu_{n}|) \,\, \frac{d s_{\sigma_{n}}}{\sigma_{n}^{1/2}}\\
&= ||\Phi_{n}||^{1/2} \int_{\gamma} \left(1+ \left(\frac{1}{|\nu_{n}|^{1/2}} -1\right) \right) (2-(1-|\nu_{n}|)) \, \, ds_{|\Phi_{0}|} \\
&=2||\Phi_{n}||^{1/2}l_{|\Phi_{0,n}|}(\gamma) + O(||\Phi_{n}||^{1/2} (1-|\nu_{n}|)),
\end{align*}
recalling that in order to obtain the metric $\sigma e$, one has to divide both hyperbolic surfaces by twice the energy, which is approximately 4 times the $L^{1}$-norm of the Hopf differential for sufficiently large energy, independent of the Riemann surface structure (see Proposition \ref{prop:mikefirst}).
Meanwhile, a similar calculation shows that an arc of the vertical foliation of $\Phi_{n}$, say $\alpha$, has length in the target hyperbolic surface given by
\begin{align*}
l_{\rho_{1n}}(\alpha) &= \int_{\alpha} \mathcal{H}_{n}^{1/2} - \mathcal{L}_{n}^{1/2} \,\, d_{s_{\sigma_{n}}}\\
&= \int_{\alpha} \mathcal{H}_{n}^{1/2} (1-|\nu_{n}|) \, ds_{\sigma_{n}}\\
&= \int_{\alpha} \frac{||\Phi_{n}||^{1/2} |\Phi_{0,n}|^{1/2}}{\sigma_{n}^{1/2} |\nu_{n}|^{1/2}} (1-|\nu_{n}|)  \,\, ds_{\sigma_{n}} \\
&=||\Phi_{n}||^{1/2} \int_{\alpha} \frac{1-|\nu_{n}|}{|\nu_{n}|^{1/2}} \,\, ds_{|\Phi_{0,n}|} \\
&= o(\mathcal{E}_{n}^{1/2}).
\end{align*}
Noting that a horizontal arc of $\Phi_{n}$ is a vertical arc of $- \Phi_{n}$, one sees the $\lambda_{1}$ and $\lambda_{2}$ are the horizontal and vertical foliations of $\Phi_{\infty}$ (the geometric limit of $\Phi_{n}$, see \cite{McMullen}) respectively. 

To get our desired harmonic map from the flat subsurface to the two trees, notice that the above estimates show that a horizontal arc of $\Phi_{0,n}$ gets mapped close to a geodesic in the target space which is a hyperbolic surface scaled by the reciprocal of total energy. As the scaled induced metric limits to the flat metric $|\Phi_{\infty}|$, a horizontal arc of $\Phi_{\infty}$ will thus be mapped by an isometry to the tree $T_{1}$ and any vertical arc collapsed, so that the limiting map in the universal cover is given by a projection onto the leaf space of the horizontal foliation of $\Phi_{\infty}$. The same argument holds for $T_{2}$.

\end{proof}
\begin{prop}\label{prop:easy}
For any closed curve $\gamma$ on the surface $S$, one has the following pair of inequalities:
\begin{align*}
l_{X_{1,n}}(\gamma) \leq l_{\Sigma_{n}} (\gamma) \leq l_{X_{1,n}}(\gamma) + l_{X_{2,n}}(\gamma) \\
l_{X_{2,n}}(\gamma) \leq l_{\Sigma_{n}} (\gamma) \leq l_{X_{1,n}}(\gamma) + l_{X_{2,n}}(\gamma)
\end{align*}
Consequently if $t_{n} L_{\sigma_{n} e_{n}} \to \eta$ as currents, then the length spectra of $\lim_{n \to \infty} t_{n} L_{X_{i,n}}$ are well-defined and are both not identically zero. If the limiting currents are denoted $\lambda_{j}$, then
$$i(\lambda_{j}, \cdot) \leq i(\eta, \cdot).$$
\end{prop}
\begin{proof}
As the minimal surface has induced metric of the form $g_{1,n} + g_{2,n}$, where the $g_{i,n}$ is a hyperbolic metric, the left side of both inequalities is immediate. If $\gamma : [0,1] \to S$ is a closed curve on the surface, then the length of $\gamma$ with respect to the induced metric is given by
\begin{align*}
l_{\sigma_{n} e_{n}}( \gamma) &= \int_{0}^{1} \sqrt{(g_{1n}+g_{2n})(\dot{\gamma}, \dot{\gamma})} \, dt \\
& \leq \int_{0}^{1} \sqrt{g_{1n}(\dot{\gamma}, \dot{\gamma})} \, dt + \int_{0}^{1} \sqrt{g_{2n}(\dot{\gamma}, \dot{\gamma})} \, dt \\
&= l_{X_{1n}}(\gamma) + l_{X_{2n}}(\gamma),
\end{align*}
where the inequality follows from the fact that $u_{i,n}$ is the identity map, and hence the differential is the identity map, and the fact for non-negative numbers $a,b$, one has $\sqrt{a+b} \leq \sqrt{a} + \sqrt{b}$. The final comment follows from choosing a closed curve $\gamma = \gamma_{n}$ to be a $\sigma_{n} e_{n}$-geodesic and using the inequality $l_{t^{2}_{n}X_{i,n}} ([\gamma]) \leq l_{t^{2}_{n}X_{i,n}}(\gamma)$.
\end{proof}

Combining Proposition \ref{prop:easy} and Theorem \ref{thm:flatanalysis}, we obtain a necessary and sufficient condition on the pair of measured laminations $\lambda_{1}$ and $\lambda_{2}$ to determine a corresponding flat part on the mixed structure.

\begin{cor}
Let $\lambda '_{i} = \lim_{n \to \infty} X'_{i,n}/2 \mathcal{E}_{n}$ be a pair of non-zero measured laminations on a subsurface $S'$. Then the pair of laminations fill if and only if the restriction of the mixed structure $\eta$ to $S'$ is flat.
\end{cor}

\begin{proof}
If $\eta$ is flat on $S'$, the preceding theorem shows the pair of laminations are dual and hence fill. If the pair of laminations do fill, then for any third lamination $\lambda'$ one has by Propostion \ref{prop:easy} that $i(\eta, \lambda ') > 0$, so that it cannot be a lamination, and hence must be flat by definition of a mixed structure.
\end{proof}

\begin{prop}
On the subsurface $S''= S \, \backslash S' $, the laminations $\lambda_{1}$ and $\lambda_{2}$ restrict to a pair of measured laminations which have no transverse intersection. If $\lambda$ denotes the measured lamination part of the mixed structure, then $i(\lambda, \lambda_{1}) = i(\lambda, \lambda_{2}) = 0$. If $L_{|\Phi_{0,n}|} \to \eta = (S', q, \lambda')$, then $i(\lambda, \lambda ') = i(\lambda ' , \lambda_{1}) = i(\lambda ' , \lambda_{2}) =0$.
\end{prop}

\begin{proof}
By Proposition $\ref{prop:easy}$, since $i(\lambda, \lambda) =0$, one has that $i(\lambda_{1}, \lambda) = i(\lambda_{2}, \lambda) = 0$. Using the inequality again yields $i(\lambda_{1}, \lambda_{2}) \leq i(\lambda, \lambda_{2}) =0$, from which the first result follows. From Proposition $\ref{prop:mike}$, one has that the length spectrum of the approximates for $\lambda$ are no less than the corresponding approximates for $\lambda'$. Taking then a sequence of simple closed curves which approximate $\lambda_{1}$ for instance yields the desired conclusion. The same follows for $\lambda_{2}$.

\end{proof}



\subsection{From geodesic currents to metric spaces}
In this section, we construct noncompact metric spaces admitting a $\pi_{1}S$-action by isometries. 

\begin{defn}
Let $X$ and $X'$ be two metric spaces and let $\epsilon >0$. Then an \emph{$\epsilon$-approximation} between $X$ and $X'$ is a relation $R$ in $X \times X'$ that is onto, so that for every $x,y \in X$ and for every $x', y' \in X'$, the conditions $xRx'$ and $yRy'$ imply $|d_{X}(x,y) -d_{X'}(x', y')| < \epsilon$.
\end{defn}

\begin{defn}
Let $X_{n}$ be a sequence of metric spaces, each admitting an isometric action by a group $\Gamma$ and a supposed limiting metric space $X_{\infty}$ also admitting an isometric action by the same group $\Gamma$. Then we say $X_{n}$ converges to $X_{\infty}$ in the sense of Gromov-Hausdorff, if for every $\epsilon >0$ and every finite set $A \subset \Gamma$, and for every compact subset $K \subset X_{\infty}$, then for $n$ sufficiently large, there is a compact set $K_{n} \subset X_{n}$ and an $\epsilon$-approximation $R_{n}$ which is $A$-equivariant between $K_{n}$ and $K$ in the following sense: for every $x \in K$, for every $x_{n}, y_{n} \in K_{n}$, and for every $\alpha \in A$, we have that the conditions $\alpha x \in K$ and $x_{n} R_{n} x$ and $y_{n} R_{n} \alpha x$ imply $d(\alpha x_{n} , y_{n}) < \epsilon$.
\end{defn}

We construct a sequence of noncompact metric spaces $X_{n}$ with an isometric action by $\Gamma = \pi_{1} S$ as follows. Take the induced metric $(S, \sigma_{n} e_{n})$ and lift the metric to the universal cover $(\widetilde{S}, \widetilde{\sigma_{n}e_{n})}$. We will deal with the case where the induced metric converges in length spectrum to a mixed structure that is not entirely laminar (this is to ensure so that we can scale our metric spaces by total energy; for the case of a mixed structure that is entirely laminar, the same discussion holds after amending the sequence of constants). The sequence of noncompact metric spaces thus will be $X_{n} = (\widetilde{S}, \widetilde{\sigma_{n} e_{n}}/ \mathcal{E}_{n})$. The following proposition is thus clear.

\begin{prop}
The manifold $X_{n} = (\widetilde{S}, \sigma_{n} e_{n}/ \mathcal{E}_{n})$ is a noncompact metric space admitting an isometric action by the group $\Gamma = \pi_{1}S$.

\end{prop}

\begin{proof}
As $X_{n}$ itself is a noncompact Riemannian manifold with $\Gamma = \pi_{1} S$ acting on it by isometries, the result follows immediately.
\end{proof}

Up to a subsequence, the metrics $(S, \sigma_{n} e_{n} / \mathcal{E}_{n})$ will converge in length spectrum to a non-trivial mixed structure $\eta = (S', q, \lambda)$. We construct a noncompact metric space $X_{\infty} = X_{\eta}$ from the mixed structure $\eta$. Regard $\eta$ as a geodesic current on $(\widetilde{S}, g)$. To any two distinct points $x,y \in \widetilde{S}$, one can form the geodesic arc $\alpha$ connecting the two points. Let $c$ be the set of bi-infinite geodesics which intersect $\alpha$ transversely. Then the intersection number $i(\eta, \alpha)$ is given by the $\eta$-measure of $c$. This yields a pseudo-metric space coming from the geodesic current $\eta$. Notice it is possible for the intersection number to be zero, for instance if the geodesic arc is disjoint from the support of the current, or if it forms no nontransverse intersection with the support of $\eta$. Taking the quotient by identifying points which are distance 0 from each other, and then taking the metric completion, yields a noncompact metric space $X_{\infty}$. As $\Gamma= \pi_{1}S$ acted on $\eta$ equivariantly, then $\Gamma$ acts by isometries on $X_{\infty}$. For a more detailed discussion about the construction of a metric space from the data of a geodesic current, see \cite{BIPP1}.

\begin{rem}
In the setting where $\eta$ is a measured foliation, the metric space $X_{\eta}$ is a familiar one. It is a $\mathbb{R}$-tree dual to the foliation. The space is constructed by collapsing the leaves of the foliation with the distance on the tree inherited by intersection number and then completing (see \cite{MS91}). The case where $\eta$ is a non-trivial mixed structure follows the same spirit of this construction. The laminar part is tree-like, formed on the universal cover by collapsing leaves of the supported lamination and then completing. On the flat part, the metric space is formed by the product of the trees dual to the vertical and horizontal lamination of a quadratic differential whose metric is the given flat metric.
\end{rem}

The preceding discussion is summarized by the following proposition.

\begin{prop}
To any mixed structure $\eta$, the construction above gives a noncompact metric space $X_{\eta}$ admitting an isometric action by $\Gamma = \pi_{1} S$.
\end{prop}

Using the Gromov-Hausdorff topology, one has the following.

\begin{thm}
A subsequence of the metric spaces $(\widetilde{S}, \widetilde{\sigma_{n} e_{n}}/ \mathcal{E}_{n})$ converges in the sense of Gromov-Hausdorff to a noncompact metric space $X_{\eta}$ coming from a mixed structure $\eta$ acted upon by $\Gamma = \pi_{1} S$.
\end{thm}

Before presenting the proof, we record one useful fact regarding convergence of maps. This follows from work of Korevaar-Schoen.

\begin{thm}[Korevaar-Schoen, see \cite{KS97}, \cite{DDW98}]\label{thm:pullback}
Let $\widetilde{M}$ be the universal cover of a compact Riemannian manifold, and let $u_{k} : \widetilde{M} \to X_{k}$ be a sequence of maps such that:
\begin{enumerate}
\item[a.] Each $X_{k}$ is an NPC space
\item[b.] The $u_{k}$'s have uniform modulus of continuity: For each $x$, there is a monotone function $\omega(x, \cdot)$, so that $\lim_{R \to 0} \omega(x, R) =0$ and $\max_{B(x,R)} d(u_{k}(x), u_{k}(y)) \leq \omega(x, R)$.
\end{enumerate}
Then the pullback metrics $d_{u_{k}}$ converge (possibly after passing to a subsequence) pointwise, locally uniformly to a pseudometric $d_{\infty}.$

\end{thm}

\begin{proof}[Proof of Theorem 6.6]
Recall from Theorem \ref{thm:geomflat}, that on $S'$ we have uniform convergence of the induced metric to the flat metric. For the complementary subsurface, recall that metric spaces were obtained as the induced metric on the minimal surface, so that the metric came from a pull-back of a harmonic map. By Proposition \ref{prop:scale}, the scaled metric is the pull-back metric of a harmonic map with energy at most 1. Hence by Theorem \ref{thm:pullback} (see Proposition 3.7 \cite{KS97}, or Theorem 2.2 \cite{DDW98}), the metrics converge uniformly. As the lifts of the induced metrics admitted an $\pi_{1}S$-action by isometries, so does the limit.

\end{proof}

\subsection{Convergence of Harmonic maps}
Not only do the metric spaces converge in a suitable topology, the harmonic maps do as well. As we have shown in the preceding section that the domains converge in the sense of Gromov-Hausdorff to a metric space arising from a mixed structure, and as shown in work of Wolf \cite{W95}, one has that the lifts of a sequence of degenerating hyperbolic metrics, when properly scaled subconverge in the sense of Gromov-Hausdorff to $\mathbb{R}$-trees dual to a particular measured lamination in the projective class of the associated point on the Thurston boundary. Hence we have both domain and target converging in the same topology to noncompact metric spaces with isometric actions by $\Gamma = \pi_{1} S$. It is natural to expect some sort of convergence in the harmonic maps. In Wolf \cite{W95}, the domain is a fixed Riemann surface, and the target is changing. Here, we have both domain and target changing (and converging). We begin by reviewing the necessary definitions.

\begin{defn}
Let $X_{n}, X_{\infty}$ be spaces admitting an action of a group $\Gamma$ and let $(Y_{n}, d_{n})$ and $(Y_{\infty}, d_{\infty})$ be metric spaces admitting an isometric action of $\Gamma$. Suppose $f_{n} : X_{n} \to Y_{n}$ and $f_{\infty}: X_{\infty} \to Y_{\infty}$ are equivariant maps. Then we say that $f_{n}$ converges (uniformly) to $f$ if
\begin{enumerate}
\item[(i)] Both $X_{n}$ and $Y_{n}$ converge (uniformly) to $X$ and $Y$ respectively in the sense of Gromov, and
\item[(ii)] For every $\epsilon >0$, there is an $N(\epsilon)$ so that for $n > N(\epsilon)$, the $\epsilon$-approximations $R_{n}, R'_{n}$ satisfies: for every $x_{n}R_{n}x$ one has $f_{n}(x_{n})R'_{n} f(x)$.
\end{enumerate}
\end{defn}
We will require a notion of harmonic for maps between singular spaces. The following can be found in more detail from \cite{EF01}.

\begin{defn}
Let $\phi \in L^{2}_{loc}(X,Y)$. The approximate energy density is defined for $\epsilon >0$ by
$$e_{\epsilon}(\phi)(x) = \int_{B_{X}(x, \epsilon)} \frac{d_{Y}^{2}(\phi(x), \phi(x'))}{\epsilon^{m+2}} \, d \mu_{g}(x').$$
\end{defn}

\begin{defn}
The energy $E(\phi)$ of a map $\phi$ of class $L^{2}_{loc}(X,Y)$ is
$$E(\phi) = \sup_{f \in C_{c}(X, [0,1])} \left( \limsup_{\epsilon \to 0} \int_{X} f e_{\epsilon}(\phi) d\mu_{g} \right)$$
\end{defn}

\begin{defn}
A harmonic map $\phi: X \to Y$ is a continuous map of class $W^{1,2}_{loc}(X,Y)$ which is bi-locally $E$-minimizing in the sense that $X$ can be covered by relatively compact subdomains $U$ for each of which there is an open set $V \supset \phi(U)$ in $Y$ such that
$$E(\phi|_{U}) \leq E(\psi |_{U})$$
for every continuous map $\psi \in W_{loc}^{1,2}(X,Y)$ with $\psi(U) \subset V$ and $\psi= \phi$ in $X \backslash U$.
\end{defn}

In the setting where both singular spaces are finite metric graphs, the resulting harmonic maps are \emph{affine maps}. Each edge of the domain graph is mapped via the constant map, or mapped linearly to the target graph. The following result of Lebeau characterizes all such harmonic maps.

\begin{thm}[Lebeau \cite{Lebeau}]\label{thm:lebeau}
Given two finite metric graphs $G$ and $G'$, every continuous map between $G$ and $G'$ is homotopic to a affine map which minimizes the energy within its homotopy class. Furthermore, the map is unique up to parallel transport.
\end{thm}

\begin{prop}\label{prop:jsgraph}
Suppose $L_{\sigma_{n} e_{n}/C_{n}}$ converges to $\lambda$, where $\lambda$ is a Jenkins-Strebel lamination. Then the sequence of metric spaces $(S, \sigma_{n}e_{n}/\mathcal{C}_{n})$ converges geometrically to a finite metric graph.
\end{prop}

\begin{proof}
This follows immediately from Theorem \ref{thm:pullback} (see also Proposition 3.7 of \cite{KS97}), as the induced metrics are the pullback metrics of a harmonic map from $\mathbb{H}^{2}$ to $\mathbb{H}^{2} \times \mathbb{H}^{2}$, which is NPC. The assumption on the modulus of continuity follows from the bound on the total energy of the maps $u_{n}$ to the rescaled target, so that total energy is at most 1. Hence, the limiting metric space is the dual graph of $\lambda$, which is a finite metric graph.
\end{proof}

\begin{thm}\label{thm:mapconverges}
Let $C_{n} \to \infty$, so that $L_{\sigma_{n}e_{n}/C_{n}} \to \eta$, where $\eta$ is a mixed structure with laminar part supported on a finite collection of simple closed curves. Suppose $L_{X_{i,n}/C_{n}} \to \lambda_{i}$, where $\lambda_{i}$ are measured also supported on a finite collection of simple closed curves.  Then the sequence of harmonic maps $u_{i,n}: (S, \sigma_{n}e_{n}/C_{n}) \to X_{i,n}/C_{n}$ converges to a harmonic map $u_{i}: X_{\eta} \to T_{i}$.
\end{thm}

\begin{proof}
Recall $X_{\eta}$ is the metric completion of the metric space obtained from the geodesic current $\eta$ by creating a pseudo-metric space from the intersection number with $\eta$, and then identifying points with 0 distance.

As the case where $\eta$ is flat has been previously handled in Theorem \ref{thm:flatanalysis}, we first construct a $\pi_{1} S$-equivariant map between the laminar part of $X_{\eta}$ and $T_{1}$ (here we will consider only the case where $\eta$ is a Strebel lamination). The same construction will produce a similar map to $T_{2}$. Let $D$ be a connected fundamental domain of the laminar region of $X_{\eta}$, then $D$ is a finite metric graph. We embed the graph $D$ into the laminar region $S''$ of the minimal surface as follows: we map each vertex of $D$ to its corresponding thick region on $S''$. The geometric convergence of the minimal surfaces to $D$ from Proposition $\ref{prop:jsgraph}$ allows us to determine which region of the minimal surface will converge to a given vertex. Once we have made our choice of where to send each vertex of $D$, if there is an edge $e$ connecting two vertices of $D$, then we send the edge $e$ to the geodesic arc connecting the two points on the minimal surface where we have mapped our two vertices. (The limiting map we will obtain later will not depend on this choice, as distances will converge uniformly.) 

As we have convergence in length spectrum and as there are only finitely many edges, we can ensure that for large $n > N(\epsilon)$, the length of the image of each edge has changed by at most $\epsilon$. We require that the embedding is proportional to arclength. Then there is a collection of continuous maps $\phi_{n}: D \to X_{n}$ with the property that given $\epsilon > 0$, there is an $N=N(\epsilon)$ so that $\phi_{n}$ is a $(1+ \epsilon)$ quasi-isometry. 

Likewise, as $\widetilde{X}_{1,n}/C_{n}$ converges geometrically to an $\mathbb{R}$-tree, a fundamental domain of  $\widetilde{X}_{1,n}/C_{n}$ will converge geometrically to a finite graph $G_{1}$ (see for instance, \cite{W95}). Hence, there is a collection of continuous  maps $\psi_{n}: X_{1,n}/C_{n} \to G_{1}$ with the same property as $\phi_{n}$.  

Form the composition $g_{n} = \psi_{n} \circ u_{1,n} \circ \phi_{n}: D \to G_{1}$, where $u_{1,n} :(S, \sigma_{n}e_{n}/C_{n}) \to X_{1,n}/C_{n}$ is a harmonic map with total energy at most 1. We claim this sequence of maps $g_{n}$ is uniformly bounded and equicontinuous. Uniform boundedness is clear as the target graph $G_{1}$ is a finite graph. To see equicontinuous, we note that as $\phi_{n}$ and $\psi_{n}$ were $(1+ \epsilon)$ quasi-isometries and since there is a uniform Lipschitz constant of the maps $u_{1,n}$, as the total energy of the maps are bounded by 1, (see \cite{KS93}, Thm 2.4.6), then equicontinuity follows. Hence, by the Arzel{\`a}-Ascoli theorem, we have a subsequence $g_{k}$ converging uniformly to a map $g:D \to G_{1}$.

 We have that $g$ is harmonic as map between singular spaces, for we have uniform convergence of distances (see \cite{KS97}) between the approximate metric spaces coming from our scaled induced metrics and the limiting $\mathbb{R}$-tree. Hence all the quantities in the definitions of the approximate energy density, and the energy converge. As there is a unique energy minimizer (up to parallel transport, by Theorem \ref{thm:lebeau}) between the limiting spaces (which are finite graphs), the map $g$ must be this unique energy minimizer.  (If $g$ were not the energy minimizer, it would have larger energy than the unique energy minimizer, by say $\delta$. One could then construct an map between the approximate Riemannian manifolds, which would have energy lower than the harmonic maps $u_{1,n}$, contradicting the harmonicity of $u_{1,n}$.)
 
 From Theorem \ref{thm:flatanalysis}, we obtained a limiting harmonic map $u'$ on the flat part of $X_{\eta}$ to the tree $T_{1}$, and now we have a limiting harmonic map $g$ from the laminar part of $X_{\eta}$ to the tree $T_{1}$. Taking the union yields the desired $u: X_{\eta} \to T_{1}$. The same argument holds for $T_{2}$.

\end{proof}




\subsection{Cores of trees}

Here we review some basics of cores of $\mathbb{R}$-trees. A more detailed overview of this material may be found in \cite{Guirardel}, \cite{W95}.

For any $\mathbb{R}$-tree, a \emph{direction} at a point $x \in T$ is a connected component of $T\backslash x$. A \emph{quadrant} in $T_{1} \times T_{2}$ is the product of $\delta_{1} \times \delta_{2}$ of two directions $\delta_{1} \subset T_{1}$ and $\delta_{2} \subset T_{2}$. We will say that the quadrant is based at $(x_{1}, x_{2}) \in T_{1} \times T_{2}$, where $x_{i}$ is the base point for the direction $\delta_{i}$. 

Let $T_{1}, T_{2}$ be a pair of trees with a common group action by $\Gamma$. Let $x = (x_{1} , x_{2}) \in T_{1} \times T_{2}$ be a base point.

\begin{defn}
Consider a quadrant $Q = \delta_{1} \times \delta_{2} \subset T_{1} \times T_{2}$. Then $Q$ is said to be \emph{heavy} if there exists a sequence $\gamma_{k} \in \Gamma$ so that
\begin{enumerate}
\item[(i)] $\gamma_{k} \cdot x \in Q$
\item[(ii)] $d_{i}(\gamma_{k} \cdot x_{i}, x_{i}) \to \infty$ as $k \to \infty$ for $i=1,2$.
\end{enumerate}
Otherwise we say $Q$ is \emph{light}.
\end{defn}

We define the \emph{core} of a product of trees to be the product $T_{1} \times T_{2}$ with all light quadrants removed. 

\begin{defn}[Guirardel, \cite{Guirardel}]
The \emph{core} $\mathcal{C}$ of $T_{1} \times T_{2}$ is the subset
$$ \mathcal{C} = T_{1} \times T_{2} \backslash \left[ \bigcup_{Q\,\, \emph{light quadrant}} Q \right]. $$
Equivalently,
$$\mathcal{C} = \bigcap_{Q = \delta_{1} \times \delta_{2} \,\, \emph{light quadrant}} (\delta_{1}^{*} \times T_{2} \cup T_{1} \times \delta_{2}^{*}). $$

\end{defn}

\begin{prop}[\cite{Guirardel}]\label{prop:core}
Let $T_{1}$ and $T_{2}$ be dual to a pair of measured foliations $\lambda_{1}$ and $\lambda_{2}$, respectively. Consider the map $p_{i} : \widetilde{S} \to T_{i}$, which maps an element of $\widetilde{S}$ to the leaf of $\widetilde{\lambda_{i}}$ which contains it. Then $\mathcal{C}(T_{1} \times T_{2}) = p_{1}(\widetilde{S}) \times p_{2} (\widetilde{S})$.
\end{prop}
\begin{proof}
The result will follow from the claim that any quadrant $Q = \delta_{1} \times \delta_{2}$ in $T_{1} \times T_{2}$ is light if and only if $p_{1}^{-1}(\delta_{1}) \cap p_{2}^{-1}(\delta_{2}) = \emptyset$. It is clear that if $p_{1}^{-1}(\delta_{1}) \cap p_{2}^{-1}(\delta_{2}) = \emptyset$, then $Q$ is light, as for each point $x \in \widetilde{S}$, the orbit of $(p_{1}(x), p_{2}(x))$ does not intersect $Q$. Conversely, if $p_{1}^{-1}(\delta_{1})$ intersects $p_{2}^{-1}(\delta_{2})$, then take $U_{\delta_{i}}$ to be an open half plane in $\widetilde{S}$ with bounded Hausdorff distance from $p_{i}^{-1}(\delta_{i})$, where $U_{\delta_{i}}$ is bounded by a geodesic in $\widetilde{\lambda_{i}}$. Then as $p_{1}^{-1}(\delta_{1})$ has nonempty intersection with $p_{2}^{-1}(\delta_{2})$, then so do $U_{\delta_{1}}$ and $U_{\delta_{2}}$. Moreover, there exists a geodesic $\gamma$ intersecting the pair of geodesics bounding $U_{\delta_{1}}$ and $U_{\delta_{2}}$. Take an element $h \in \pi_{1}S$ whose axis is $\gamma$. Then $h$ is hyperbolic in both $T_{1}$ and $T_{2}$ and $h$ makes $Q$ heavy.
\end{proof}

\begin{rem}
This characterization of the core of two trees is particularly useful in our setting where the trees come from measured laminations. The map $p$ which sends $\mathbb{H}^{2}$ to the leaf space of a measured lamination is a $\pi_{1}S$-equivariant harmonic map, and as a product of harmonic maps is harmonic, we see that the core is the image of the $\pi_{1}S$-equivariant harmonic map ($p_{1} \times p_{2}) : \mathbb{H} \to T_{1} \times T_{2}$.
\end{rem}

In the setting where where $T_{1}$ and $T_{2}$ arise from two transverse measured foliations $\lambda_{1}$ and $\lambda_{2}$, then $\mathcal{C}(T_{1} \times T_{2}) / \pi_{1}S$ is isometric to $S$ endowed with the unique singular Euclidean metric whose vertical and horizontal foliations are $\lambda_{1}$ and $\lambda_{2}$.

We present our next main result concerning the relation between the mixed structures we obtain as limits of the induced metrics and the limits of the corresponding graphs of the minimal langrangians.

\begin{thm}\label{thm:core}
Suppose $C_{n} \to \infty$, so that $L_{\sigma_{n} e_{n}/C_{n}} \to \eta$ and $X_{1,n}/C_{n} \to T_{1}$ and $X_{2,n}/C_{n} \to T_{2}$. Then the metric space $X_{\eta}$ is isometric to the core of the pair of trees $(T_{1}, T_{2})$. Consequently, the minimal lagrangians $\widetilde{\Sigma_{n}}/C_{n} \subset \mathbb{H}^{2}/C_{n} \times \mathbb{H}^{2}/C_{n}$ converge geometrically to the core $\mathcal{C}(T_{1} \times T_{2}) \subset T_{1} \times T_{2}$.
\end{thm}

\begin{proof}

Define the auxiliary map $\Psi: \text{P(ML} \times\text{ML}) \to \text{PMix}(S)$ by
$$\Psi([\lambda_{1}, \lambda_{2}]) = \lim_{n \to \infty} [L_{\sigma_{n} e_{n}}],$$
where $\Sigma_{n} \subset X_{1,n} \times X_{2,n}$ is the minimal lagrangian with induced metric $2 \sigma_{n} e_{n}$ and $(X_{1,n}, X_{2,n})$ converge projectively to $[(\lambda_{1}, \lambda_{2})]$. We claim the map is well-defined.


Choose $[(\lambda_{1}, \lambda_{2})] \in$ $\text{P(ML} \times\text{ML})$ and a representative $(\lambda_{1}, \lambda_{2}) \in [(\lambda_{1}, \lambda_{2})]$. Then if both $ (X_{1,n}/k_{n} , X_{2,n}/k_{n})$ and $(Y_{1,n}/d_{n}, Y_{2,n}/d_{n})$ converge in length spectrum to $(\lambda_{1}, \lambda_{2})$, then for large enough $n$, we will have that $X_{1,n}/k_{n}$ will be close to $Y_{1,n}/d_{n}$ as negatively curved Riemannian surfaces (and likewise for $X_{2,n}/k_{n}$ and $Y_{2,n}/d_{n}$) by \cite{Otal}. Hence the induced metrics on the respective pairs of minimal langrangians will have close length spectra, so that $\Psi$ is well-defined. 

To see that $\Psi$ is continuous, observe that the induced metric on the minimal surface varies continuously as a map defined on $\mathcal{T}(S) \times \mathcal{T}(S)$, and as the length spectrum of the induced metric varies continuously as one takes a sequence of hyperbolic surfaces $(X_{1,n}, X_{2,n}) \to [(\lambda_{1}, \lambda_{2})] \in \text{P(ML} \times\text{ML})$, one finds the space of mixed structures varies continuously on $\text{P(ML} \times\text{ML})$ by a diagonal argument.

But we now have a harmonic map from $X_{\eta}$ to $T_{1} \times T_{2}$. From Theorem \ref{thm:flatanalysis}, the harmonic map on the flat part is given by projection to its vertical and horizontal lamination. By Theorem \ref{thm:mapconverges}, the harmonic map from the laminar part is given by an affine map, when both trees come from Jenkins-Strebel differentials. 

As the homotopy class of the maps were given by the identity map, one sees that vertices on the domain graph are mapped to the vertices of the target graph (the thick regions of the minimal surface are necessarily mapped to the thick regions of the target scaled hyperbolic surface; for if a vertex were to be mapped away from vertices, the approximating thick region of the minimal surface would be mapped deep into a thin region of the target scaled hyperbolic surface, so that the thick region of the minimal surface would not have diameter going to zero, contradicting the geometric convergence of the thick region to a vertex). Hence by Theorem \ref{thm:lebeau}, the map is an affine map which maps vertices to the corresponding vertices.



But this yields the product metric for the core of the two trees (see Proposition \ref{prop:core} and the remark which follows). The equality of the mixed structure and the core of the trees then holds for pairs of $\mathbb{R}$-trees dual to a pair of Jenkins-Strebel foliations, which is a dense set in $\text{P(ML} \times\text{ML})$, and both quantities vary continuous for $\text{P(ML} \times\text{ML})$, thus the theorem follows.

\end{proof}

\begin{rem}
In fact, by Theorems \ref{thm:flatanalysis} and  \ref{thm:mapconverges}, the sequence of $\rho$-equivariant harmonic maps from $\mathbb{H}^{2}$ to $\mathbb{H}^{2} \times \mathbb{H}^{2}$  converges projectively to a harmonic map from $\mathbb{H}^{2}$ to the product of $\mathbb{R}$-trees, whose image is the core of the trees.
\end{rem}

\section{Applications to maximal surfaces in AdS$^{3}$}
In this section, we prove the required analogues of the minimal lagrangian setting to show a similar result for limits of maximal surfaces.

\begin{prop}
On a fixed hyperbolic surface $(S, \sigma)$ one has $\mathcal{H}_{1} = \mathcal{H}_{2}$ if and only if $e_{1} = e_{2}$.
\end{prop}

\begin{proof}
Note that if $e_{1} = e_{2}$ then $|\Phi_{1}| = |\Phi_{2}|$ by Lemma $\ref{lem:zeno}$. From $|\Phi_{1}| = |\Phi_{2}|$, one has by some basic algebra $\mathcal{L}_{2} = \frac{\mathcal{H}_{1} \mathcal{L}_{1}}{\mathcal{H}_{2}}$. From the Bochner formula, one has
\begin{align*}
\Delta \log \mathcal{H} &= 2 \mathcal{H} -2 \mathcal{L} -2 \\
\frac{1}{2} \Delta \log \frac{\mathcal{H}_{1}}{\mathcal{H}_{2}} & = (\mathcal{H}_{1} - \mathcal{H}_{2}) - (\mathcal{L}_{1} - \mathcal{L}_{2})\\
&=(\mathcal{H}_{1} - \mathcal{H}_{2}) - \mathcal{L}_{1}(1- \frac{\mathcal{H}_{1}}{\mathcal{H}_{2}}).
\end{align*}
At a point $p \in S$ for which the quotient $\mathcal{H}_{1}/\mathcal{H}_{2}$ achieves its maximum (which without loss of generality we may assume to be greater than 1, or else as before we may reindex), the left hand side of the preceding calculation must be non-positive, but the right hand side is positive, hence $\mathcal{H}_{1} = \mathcal{H}_{2}$ everywhere.
\end{proof}

\begin{prop}\label{prop:hscale}
On a fixed hyperbolic surface $(S, \sigma)$ if $\mathcal{H}_{1} = c \mathcal{H}_{2}$ then $c=1$.
\end{prop}

\begin{proof}
Without loss of generality, take $c >1$ or we we may reindex to ensure this is the case. Once again by the Bochner formula,
\begin{align*}
\Delta \log \frac{\mathcal{H}_{1}}{\mathcal{H}_{2}} &= 2(\mathcal{H}_{1} - \mathcal{H}_{2}) -2 (\mathcal{L}_{1} - \mathcal{L}_{2})\\
0 = \Delta \log c &= 2 (c \mathcal{H}_{2} - \mathcal{H}_{2}) - 2(\mathcal{L}_{1} - \mathcal{L}_{2})\\
&= 2 \mathcal{H}_{2}(c-1) - 2( \mathcal{L}_{1} - \mathcal{L}_{2})
\end{align*}
Hence, everywhere one has
$$\mathcal{L}_{1} - \mathcal{L}_{2} = \mathcal{H}_{2} (c-1) > 0.$$
But $\mathcal{L}_{1}$ vanishes at the zeros of the quadratic differential $\Phi_{1}$, a contradiction. Hence $c=1$.
\end{proof}

\begin{prop}\label{prop:holenergy}
Let $H = \int \mathcal{H} \, dA(\sigma)$. Then $\mathcal{E} = H +4 \pi \chi$. Consequently if $\mathcal{E}_{n} \to \infty$, then $\lim_{n \to \infty} \mathcal{E}_{n} / H_{n} =2$.
\end{prop}

\begin{proof}
As $\mathcal{J} = \mathcal{H} - \mathcal{L}$ and $\int \mathcal{J} \sigma dz d \overline{z} = -2 \pi \chi$, one has
$$ \int \mathcal{H} \sigma dz d \overline{z} + 2 \pi \chi = \int \mathcal{L} \sigma dz d \overline{z}.$$
Adding the terms yields
$$ \mathcal{E} = \int (\mathcal{H}+ \mathcal{L}) \sigma dz d \overline{z} = 2 \int \mathcal{H} \sigma dz d \overline{z} + 4 \pi \chi = 2 H + 4 \pi \chi.$$
\end{proof}

Recall from Section 2.6, the existence and uniqueness of a spacelike, embedded maximal surface in any GHMC AdS$^{3}$ manifold. 

\begin{prop}[Lemma 3.6 \cite{KS07}]
The induced metric on the maximal surface is of the form $\mathcal{H} \sigma$.
\end{prop}

\begin{prop}
The induced metric on the maximal surface has strictly negative curvature.
\end{prop}

\begin{proof}
The formula for curvature is given by
\begin{align*}
K_{\mathcal{H} \sigma} &= - \frac{1}{2\mathcal{H} \sigma} \Delta \log \mathcal{H} \sigma \\
&= -\frac{1}{2} \frac{1}{\mathcal{H}} \left( \frac{\Delta \log \mathcal{H}}{\sigma} + \frac{ \Delta \log \sigma}{\sigma} \right)\\
&= \frac{- \mathcal{J}}{\mathcal{H}}
\end{align*}
where the last step comes from the Bochner equation and the curvature of the hyperbolic metric.
\end{proof}

\begin{thm}
There exists an embedding of the space of maximal surfaces into the space of projectivized currents.
\end{thm}

\begin{proof}
As the induced metrics on the maximal surfaces are negatively curved, they may be realized as geodesic currents. By Proposition \ref{prop:hscale}, the projectivization remains injective.
\end{proof}

\begin{thm}
The closure of the space of induced metrics on the maximal surfaces is given by the space of flat metrics arising from unit norm holomorphic quadratic differentials and projectivized mixed structures.
\end{thm}

\begin{proof}
To any induced metric $\mathcal{H} \sigma$ on the maximal surface, there is a unique singular quadratic differential metric $|\Phi|$ associated to it. Some algebra shows that
$$ \mathcal{H} \sigma = \frac{|\Phi |}{|\nu|} \geq |\Phi|,$$
which for high energy, Proposition \ref{prop:holenergy} tells us $H$ approximates the $L^{1}$-norm of the quadratic differential, so that if the sequence of unit-norm quadratic differentials converges to measured lamination, then so does the projective current associated to the induced metric on the maximal surface. Hence, we assume the sequence of unit-norm quadratic differential metrics converges to a mixed structure. On the flat part of the mixed structure, we know that up to a subsequence the Beltrami differentials converges uniformly to 1 outside of a small region about the zeros of the differential and a cylindrical neighborhood of the boundary curves. But then we know that on this subsurface the maximal surface metric will converge to $|\Phi_{\infty}|$ in terms of its length spectrum. As the total area of the mixed structure is 1 and we have normalized the maximal surface metric by the total holomorphic energy, on the complement, the area of the metric tends to 0, so that the restriction of the limiting current is a measured lamination.

\end{proof}

We observe there is a rather interesting trichotomy at play here. For high energy, on the subsurface $S'$, if the quadratic differentials converge to $|\Phi_{\infty}|$ then so do the associated sequence of minimal surface metrics and the sequence of maximal surface metrics.

\section{Compactification of maximal representations to $\text{PSL}(2, \mathbb{R}) \times \text{PSL}(2, \mathbb{R})$}

In this final section, we provide an application of our work to compactifying the maximal component of the representation variety $\chi(\text{PSL}(2, \mathbb{R}) \times \text{PSL}(2, \mathbb{R}))$. The theory of maximal representations is defined for general Hermitian Lie groups $G$ and is considerably more straightforward to define in our specific setting of $G = \text{PSL}(2, \mathbb{R}) \times \text{PSL}(2, \mathbb{R})$. Nevertheless, we will define a maximal representation in the general setting before providing a straightforward characterization in our setting.

Let $G$ be a \emph{Hermitian} Lie group, that is a noncompact simple Lie group whose symmetric space $G/K$ is a K{\"a}hler manifold. Equivalently, there is a $G$-invariant two-form $\omega$ on $G/K$. Let $S$ be a closed, orientable, smooth surface of genus $g \geq 2$. Then given a representation $\rho: \pi_{1} S \to G$, one can choose any $\rho$-equivariant map $\tilde{f}: \widetilde{S} \to G/K$ and define the \emph{Toledo invariant} to be
$$T(\rho) : = \frac{1}{2 \pi} \int_{S} \tilde{f}^{*} \omega .$$

The Toledo invariant will be well-defined for each such representation as the number obtained will not depend on the choice of $\tilde{f}$ chosen above. A well-known \emph{Milnor-Wood} type inequality holds for the Toledo invariant,
$$|T(\rho)| \leq |\chi(s)| \cdot \text{rank}(G/K).$$

Representations whose Toledo invariant attains the upperbound are known as \emph{maximal representations}. We now restrict our attention specifically to the group $G =$ PSL(2, $\mathbb{R}) \times$ PSL(2, $\mathbb{R})$, whose associated symmetric space is $\mathbb{H}^{2} \times \mathbb{H}^{2}$.

To each representation to the group $\text{PSL}(2, \mathbb{R}) \times \text{PSL}(2, \mathbb{R})$, one obtains a pair of representations to the group PSL$(2, \mathbb{R})$. By work of Goldman \cite{G}, the Euler number of representations to PSL$(2,\mathbb{R})$ characterizes the connected components of the representation variety. The maximal representations are precisely those whose projections live in the Hitchin component of PSL$(2, \mathbb{R})$ representations, that is, those representations that are both discrete and faithful. Hence, such a representation yields a pair of points in Teichm{\"u}ller space and an associated minimal surface. We may parameterize such representations by the induced metric on the minimal surface $\Sigma$, as well as the $\rho = (\rho_{1}, \rho_{2})$-equivariant harmonic map from $\mathbb{H}^{2}$ to $\mathbb{H}^{2} \times \mathbb{H}^{2}$, given by the graph of the minimal langrangian from Theorem \ref{thm:Schoen}. As a final consequence of our study of these minimal langrangians, we obtain a compactifcation of the the maximal component of surface group representations to $\text{PSL}(2, \mathbb{R}) \times \text{PSL}(2, \mathbb{R})$.

\begin{thm}
Let $S$ be a closed surface of genus $g >1$. The space of maximal representations of $\pi_{1}(S)$ to \emph{PSL}$(2,\mathbb{R})$ $\times$ \emph{PSL}$(2,\mathbb{R})$ embeds into the space of $\pi_{1}S$-equivariant harmonic maps from $\mathbb{H}^{2} \to \mathbb{H}^{2} \times \mathbb{H}^{2}$, whose graphs are minimal lagrangians. The Gromov-Hausdorff limits of these maps are given by harmonic maps from $\mathbb{H}^{2}$ to $T_{1} \times T_{2}$, where $T_{1}$ and $T_{2}$ are a pair of $\mathbb{R}$-trees coming from a projective pair of measured foliations, with image given by the core of the trees.

\end{thm}

\begin{proof}
To any maximal representation $\rho = (\rho_{1}, \rho_{2})$, we may look at the two closed hyperbolic surfaces given by $X_{1} = \mathbb{H}^{2} \, \backslash \rho_{1} $ and $X_{2} = \mathbb{H}^{2} \, \backslash \rho_{2}$. This gives a clear homeomorphism between the maximal component and two copies of Teichm{\"u}ller space. By Theorem \ref{thm:Schoen}, we obtain a minimal lagrangian between $X_{1}$ and $X_{2}$ which respects the marking. If $\Sigma$ denotes the conformal structure of the graph, then the inclusion map $i: \Sigma \to X_{1} \times X_{2}$ is a conformal map, which lifts to the desired $\rho$-equivariant map from $\mathbb{H}^{2}$ to $\mathbb{H}^{2} \times \mathbb{H}^{2}$. The map which associates the representation $\rho$ to this map is continuous and is injective as distinct representations have distinct minimal lagrangians, hence yielding our desired embedding. 

If $\rho_{n}$ is a sequence of representations leaving all compact sets, then either $g_{1,n}$ or $g_{2,n}$  (or both) leaves all compact sets in Teichm{\"u}ller space (recall $(S,g_{i}) = \mathbb{H}^{2}/\rho_{i})$. By Theorem \ref{thm:compact}(up to a subsequence) the sequence of induced metrics on the graphs converge projectively to a mixed structure. Let $c_{n}$ be the seqeunce of constants for which we divide the induced metric to ensure length spectrum convergence to a non-zero mixed structure with self-intersection 1 or a measured lamination. If we scale the target by the same sequence of constants, then the total energy of the sequence of harmonic maps is now uniformly bounded, so that by Theorem \ref{thm:pullback}, the maps converge to a map from $\mathbb{H}^{2}$ to $T_{1} \times T_{2}$, where $T_{i}$ is the $\mathbb{R}$-tree associated to the limit $\widetilde{X}_{i,n}/c_{n}$ (notice that $T_{i}$ may be a single point). By Theorem \ref{thm:core}, the image is given by the core of the trees, which suffices for the proof.

\end{proof}

\end{document}